\documentclass[reqno,11pt]{article}
\usepackage{mathrsfs}
\usepackage{amssymb}
\usepackage{amsthm}
\usepackage{amsmath}
\usepackage{amsfonts}
\usepackage{mathtools}
\usepackage{enumerate}
\usepackage{bm}

\usepackage{graphicx}

\usepackage{bbm}
\usepackage{mathrsfs}
\usepackage{amssymb}
\usepackage[thicklines]{cancel}

\UseRawInputEncoding

\usepackage[usenames,dvipsnames]{xcolor}
\usepackage{soul}

\usepackage{txfonts}

\usepackage{psfrag}
\usepackage[numbers,sort&compress]{natbib}
\usepackage{setspace,color,wrapfig}
\usepackage[colorlinks,linkcolor=blue,citecolor=cyan]{hyperref}

\numberwithin{equation}{section}
\newtheorem{theorem}{Theorem}[section]
\newtheorem{lemma}[theorem]{Lemma}
\newtheorem{corollary}[theorem]{Corollary}
\newtheorem{proposition}[theorem]{Proposition}
\newtheorem{problem}[theorem]{Problem}
\theoremstyle{definition}
\newtheorem{definition}[theorem]{Definition}

\newtheorem{remark}[theorem]{Remark}
\theoremstyle{remark}

\begin{document}
\title{Supersonic  Euler-Poisson flows with nonzero vorticity  in convergent nozzles  }
\author{Yuanyuan Xing\thanks{School of Mathematics and Statistics,
                Northeastern University at Qinhuangdao, Qinhuangdao,  Hebei Province, 066004,  China. Email: xingyuanyuan@neuq.edu.cn }\and
Zihao Zhang\thanks{School of Mathematics, Jilin University, Changchun, Jilin Province, 130012, China. Email: zhangzihao@jlu.edu.cn}}
\date{}
\newtheorem{pro}{Problem}[section]
 \newtheorem{thm}{Theorem}[section]
 \newtheorem{cor}[thm]{Corollary}
 \newtheorem{lem}[theorem]{Lemma}
 \newtheorem{prop}[thm]{Proposition}
 \newtheorem{defn}{Definition}[section]
  \theoremstyle{remark}
 \newtheorem{rem}{\bf{Remark}}[section]
 \numberwithin{equation}{section}

\def\be{\begin{equation}}
\def\ee{\end{equation}}
\def\beq{\begin{equation*}}
\def\eeq{\end{equation*}}
\def\bc{\begin{cases}}
\def\ec{\end{cases}}
\def\bs{\begin{split}}
\def\es{\end{split}}
\def\oT{\overline{T}}
\def\hT{\hat{T}}
\def\d{{\rm div}}
\def\R{\mathbb{R}^2}
\def\c{\cdot}
\def\n{\nabla}
\def\i{\infty}
\def\D{\mathcal{D}(u)}
\def\A{\mathcal{P}}
\def\p{\partial}

\renewcommand\figurename{\scriptsize Fig}
\pagestyle{myheadings} \markboth{ Supersonic Euler-Poisson flows  with
nonzero  vorticity }{ Supersonic Euler-Poisson flows with
nonzero  vorticity }\maketitle
\begin{abstract}
This paper concerns  supersonic flows with nonzero vorticity governed by the steady Euler-Poisson system, under the coupled effects of the electric potential and the geometry of a convergent nozzle. By the coordinate rotation, the existence of  radially symmetric supersonic flows is proved. We then establish the structural stability of   these background  supersonic flows  under multi-dimensional perturbations of the boundary conditions. One of the crucial   ingredients of
the analysis  is the reformulation of the steady  Euler-Poisson system into a deformation-curl-Poisson system and several transport equations via the deformation-curl-Poisson decomposition. Another  one is to obtain  the well-posedness of the boundary value problem for the associated linearized hyperbolic-elliptic coupled system, which is achieved through a delicate choice of multiplier to derive a priori estimates. The result indicates that the electric field force in compressible flows can counteract the  geometric effects of the convergent nozzle and thereby stabilize key  physical features of the flow.
\end{abstract}
\begin{center}
\begin{minipage}{5.5in}
Mathematics Subject Classifications 2020: 35G60, 35J66, 35M32, 35Q35, 76J20.\\
Key words:  supersonic Euler-Poisson flows, vorticity, hyperbolic-elliptic system,    stability.
\end{minipage}
\end{center}
\section{Introduction}\noindent
\par Flow motion governed by a self-generated electric field, such as electron transport in submicron semiconductor devices \cite{Markowich}, ion transport through channel proteins \cite{Chen}, and other related physical phenomena, can be modeled by the Euler-Poisson system
\begin{equation}\label{ep}
\begin{cases}
\rho_t+\mathrm{div}(\rho \mathbf{u})=0,\\
(\rho\mathbf{u})_t+\mathrm{div}(\rho \mathbf{u}\otimes\mathbf{u}+P\mathbf{I})=\rho \nabla\Phi,\\
(\rho\mathcal{E})_t+\mathrm{div}(\rho\mathcal{E}\mathbf{u}+P\mathbf{u})=\rho\mathbf{u}\cdot \nabla\Phi,\\
\Delta\Phi=\rho-b({\bf x}).
\end{cases}
\end{equation}
In the hydrodynamical model of semiconductor devices or plasmas,  $\mathbf{u}$, $\rho$, $P$ and $\mathcal{E}$ represent the macroscopic particle velocity, electron density, pressure and total energy, respectively. The electric potential $\Phi$ is generally generated by the Coulomb force of particles.  $b>0$ denotes the density of fixed, positively charged background ions. $\mathbf{I}$ is an $n\times n$ identity matrix. The system \eqref{ep} is closed  with the aid of definition of specific total energy and the equation of
state
\begin{equation*}
\mathcal{E}=\frac{\vert\mathbf{u}\vert^2}{2}+\frac{P}{(\gamma-1)\rho},\quad P=P(\rho,S)=\mathrm{e}^S\rho^{\gamma},
\end{equation*}
provided $\gamma>1$ is the adiabatic exponent and $S$ is the entropy. In this paper, we focus on the two-dimensional steady Euler-Poisson system in a convergent nozzle
\begin{equation*}
\tilde\Omega=\Big\{{\bf x}=(x_1,x_2):0<r_1<\sqrt{x_1^2+x_2^2}<r_2<+\infty, -\theta_0<\arctan\frac{x_2}{x_1}<\theta_0\Big\}
\end{equation*}
as follows:
\begin{equation}\label{ep1}
\begin{cases}
(\rho u)_{x_1}+(\rho v)_{x_2}=0,\\
(\rho u^2)_{x_1}+(\rho uv)_{x_2}+P_{x_1}=\rho \Phi_{x_1},\\
(\rho uv)_{x_1}+(\rho v^2)_{x_2}+P_{x_2}=\rho \Phi_{x_2},\\
(\rho u\mathcal{E}+pu)_{x_1}+(\rho v\mathcal{E}+pv)_{x_2}=\rho \mathbf{u}\cdot\nabla\Phi,\\
\Delta\Phi=\rho-b({\bf x}),
\end{cases}
\end{equation}
provided $\theta_0\in\left(0,\frac{\pi}{2}\right)$, $r_1$, $r_2$ are fixed positive constants. The local sound speed $c$ and the Mach number $M$ are given by $c=\sqrt{P_{\rho}(\rho,S)}$ and $M=\frac{\vert \mathbf{u}\vert}{c}$, respectively. If $M>1$, the system \eqref{ep1} can be decomposed into a nonlinear hyperbolic-elliptic coupled system with homogeneous transport equations, and the flow is called supersonic. If $M<1$, then \eqref{ep1} reduces to a nonlinear elliptic system with homogeneous transport equations, and the flow is subsonic. For a transonic state, the system typically involves degenerate surfaces, type changes, strong nonlinearity and nonlocal effects. Although transonic shock problems have been investigated for the one-dimensional and quasi-one-dimensional Euler-Poisson system (see \cite{Ascher,Gamba,Luo1,Luo2,Rosini,Yeh} and the references therein), there are very few known results about  multi-dimensional transonic Euler-Poisson flows. A recent development is the result in \cite{Bae8}, which established the existence and structural stability of smooth transonic flows for the steady Euler-Poisson system in a two-dimensional nozzle.

In previous works, the  subsonic and supersonic solutions for the multi-dimensional Euler-Poisson system have been studied extensively. When the current flux is sufficiently small, the existence of subsonic solutions was proved in \cite{Degond1,Degond2,Guo,Markowich0,Yeh}. The unique existence and stability of isentropic irrotational subsonic solutions with large variations for the Euler-Poisson system under perturbations of exit pressure were demonstrated in \cite{Bae3}, where a special structure of the associated linearized elliptic system was discovered to derive a priori estimate. The unique existence
of two-dimensional and three-dimensional axisymmetric subsonic flows with nonzero vorticity in a
finite nozzle was established in \cite{Bae1,Bae5} by
using the Helmholtz decomposition.  Similar results for self-gravitating isentropic and non-isentropic subsonic flows in two-dimensional flat nozzles and annular domains were obtained in \cite{Bae4,Cao,Duan,WZ25} via the stream function formulation.
   Recently, the authors in \cite{XZ25}   establish the existence and uniqueness of smooth
subsonic Euler-Poisson flows with nonzero vorticity in a two-dimensional convergent nozzle.

Compared to the subsonic case, the theory of supersonic solutions remains relatively underdeveloped. The structural stability of supersonic irrotational flows and flows with nonzero vorticity in two-dimensional flat nozzles was established in \cite{Bae1}, where a priori estimate for the associated linearized hyperbolic-elliptic system was derived through a delicate choice of the multiplier. Supersonic potential flows in three-dimensional cylinders and two-dimensional divergent nozzles were investigated in \cite{Bae6,Duan0,WZ25-1,WZ25-2}. Owing to the presence of the Poisson equation, the system \eqref{ep1} in the supersonic regime becomes a second order quasilinear hyperbolic-elliptic coupled system, even for isentropic irrotational flows. Consequently, the question on well-posedness theory for the steady Euler-Poisson system for supersonic solutions presents substantially greater analytical challenges than its subsonic case.

The main goal of this paper is to establish the existence and  stability of supersonic flows governed by the full Euler-Poisson system with nonzero vorticity, under the coupled effects of the electric potential and the geometry of a convergent nozzle. Prior work \cite{Liu} shows that transonic shock flows described by Euler system are dynamically stable in divergent nozzles but unstable in convergent ones. Notably, \cite{Luo1} and \cite{Luo2} reveal that when the background charge density is less than the sonic density, flows in one dimensional flat nozzles governed by the Euler-Poisson system exhibit behavior similar to that in divergent nozzles for the Euler system. These observations show that the geometry of a convergent domain exerts a destabilizing effect on the analytical behavior of nozzle flows. Our results show that the electric field force in compressible flows can offset this destabilizing geometric influence, thereby stabilizing key physical features of the flow.

A key element of our analysis is the introduction of a new coordinate rotation, motivated by the goal of constructing solutions propagating in the direction of $-\mathbf{e}_{\tilde{r}}$. Due to the destabilizing effect of the convergent geometry, the well-posedness of supersonic solutions in such a nozzle setting becomes inherently nontrivial. In addition, for flows with nonzero vorticity, the traditional Helmholtz decomposition of the velocity field $\mathbf{u}=\nabla \varphi+\nabla^{\perp}\psi$ with $\perp=(\partial_{x_2},-\partial_{x_1})$, requires a more delicate and technically involved treatment to ensure the unique solvability of the vector potential $\varphi$ and scalar potential $\psi$ in polar coordinates. To overcome this difficulty, we employ a deformation-curl-Poisson decomposition for the Euler-Poisson system, originally introduced in \cite{Weng}. The advantage of this decomposition is that the system can be written as a deformation-curl-Poisson system for the velocity field and the electric potential, together with two homogeneous transport equations for the entropy and the pseudo-Bernoulli's quantity. Within this framework, the vorticity is resolved by an algebraic equation for the pseudo-Bernoulli's function and the entropy. Furthermore, this deformation-curl-Poisson system can be weakly decomposed into a second order hyperbolic-elliptic system for the electric potential function and another associated potential function. To establish the well-posedness of the boundary value problem for the linearized hyperbolic-elliptic coupled system, a critical step is the delicate choice of multiplier to gain the energy estimate.

The rest of this paper is organized as follows. In Section 2, we formulate the problem in detail and state the main result. In Section 3, we employ the deformation-curl-Poisson decomposition to derive a second order linearized hyperbolic-elliptic coupled system and two transport equations from the Euler-Poisson system around radially symmetric supersonic solutions. Furthermore, to  obtain the well-posedness of the linearized
problem,  a key ingredient is a priori estimates given in Proposition \ref{prop1}. In Section 4, we establish the nonlinear structural stability of supersonic flows with nonzero vorticity for the Euler-Poisson system under multi-dimensional perturbations of boundary conditions. The proof mainly relies on the iteration method and the estimates  for the  hyperbolic-elliptic system and the transport equations.

\section{The nonlinear problem and main results}\noindent
\par Introduce the polar coordinates $(\tilde r, \tilde\theta)$:
 \begin{eqnarray*}
\tilde r=\sqrt{x_1^2+x_2^2},\quad \tilde\theta=\arctan \frac{x_2}{x_1}.
\end{eqnarray*}
Assume that the velocity, the density, the pressure, the electric potential and the ion background density are of the form
\begin{equation*}
\begin{aligned}
{\bf u}({\bf x})= \tilde U(\tilde r,\tilde \theta)\mathbf{e}_{\tilde r}+ \tilde V(\tilde r,\tilde\theta)\mathbf{e}_{\tilde\theta}, \ \
\rho({\bf x})=\tilde\rho(\tilde r,\tilde \theta),\ \ P({\bf x})=\tilde P(\tilde r,\tilde \theta),\ \ \Phi({\bf x})=\tilde \Phi(\tilde r,\tilde \theta),\ \ b({\bf x})=\tilde b(\tilde r,\tilde \theta),
\end{aligned}
\end{equation*}
  where
\begin{align*}
{\bf e}_{\tilde r}=\begin{pmatrix} \cos\theta\\ \sin\theta\end{pmatrix},\quad
{\bf e}_{\tilde\theta}=\begin{pmatrix}-\sin\theta\\ \cos\theta\end{pmatrix}.
\end{align*}
Then the system \eqref{ep1} in the polar coordinates takes the following form
\begin{equation}\label{EP}
\begin{cases}
\begin{aligned}
&(\tilde{r}\tilde{\rho} \tilde{U})_{\tilde{r}}+(\tilde{\rho} \tilde{V})_{\tilde{\theta}}=0,\\
&\tilde{\rho} \tilde{U}\tilde{U}_{\tilde{r}}+\frac{1}{\tilde{r}}\tilde{\rho} \tilde{V}\tilde{U}_{\tilde{\theta}}-\frac{1}{\tilde{r}}\tilde{\rho} \tilde{V}^2+\tilde{P}_{\tilde{r}}=\tilde{\rho}\tilde{\Phi}_{\tilde{r}},\\
&\tilde{\rho} \tilde{U}\tilde{V}_{\tilde{r}}+\frac{1}{\tilde{r}}\tilde{\rho} \tilde{U}\tilde{V}+\frac{1}{\tilde{r}}\tilde{\rho} \tilde{V}\tilde{V}_{\tilde{\theta}}+\frac{1}{\tilde{r}}\tilde{P}_{\tilde{\theta}}
=\frac{1}{\tilde{r}}\tilde{\rho}\tilde{\Phi}_{\tilde{\theta}},\\
&(\tilde{\rho} \tilde{U}\tilde{\mathscr{B}})_{\tilde{r}}+\frac{1}{\tilde{r}}(\tilde{\rho} \tilde{V}\tilde{\mathscr{B}})_{\tilde{\theta}}+\frac{1}{\tilde{r}}\tilde{\rho} \tilde{U}\tilde{\mathscr{B}}
=\tilde{\rho}(\tilde{U}\tilde{\Phi}_{\tilde{r}}+\frac{1}{\tilde{r}}\tilde{V}\tilde{\Phi}_{\tilde{\theta}}),\\
&\tilde{\Phi}_{\tilde{r}\tilde{r}}+\frac{1}{\tilde{r}^2}\tilde{\Phi}_{\tilde{\theta}\tilde{\theta}}
+\frac{1}{\tilde{r}}\tilde{\Phi}_{\tilde{r}}=\tilde{\rho}-\tilde b,
\end{aligned}
\end{cases}
\end{equation}
where the Bernoulli function $\tilde{\mathscr{B}}$ is given  by
\begin{equation}
\tilde{\mathscr{B}}=\frac{\tilde{U}^2+\tilde{V}^2}{2}+\frac{\gamma \tilde{P}}{(\gamma-1)\tilde{\rho}}.
\end{equation}
Furthermore, we define the pseudo-Bernoulli function $\tilde{\mathscr{K}}$  by
\begin{equation*}
\tilde{\mathscr{K}}=\tilde{\mathscr{B}}-\tilde{\Phi}.
\end{equation*}
Then it is derived from \eqref{EP} that
\begin{equation}\label{EP-k}
(\tilde{\rho} \tilde{U}\tilde{\mathscr{K}})_{\tilde{r}}+\frac{1}{\tilde{r}}(\tilde{\rho} \tilde{V}\tilde{\mathscr{K}})_{\tilde{\theta}}+\frac{1}{\tilde{r}}\tilde{\rho} \tilde{U}\tilde{\mathscr{K}}
=0.
\end{equation}
\par As we seek solutions flowing in the direction of $-\mathrm{e}_{\tilde{r}}$, introduce new variables $(r,\hat{r},\theta)=(r_2-\tilde{r},\tilde{r},\tilde{\theta})$ and set
\begin{equation}
(\rho, U, V, P, \Phi,b)(r,\theta):=(\tilde{\rho}, -\tilde{U}, \tilde{V}, \tilde{P}, \tilde{\Phi},\tilde b)(\tilde{r},\tilde{\theta}),
\end{equation}
\eqref{EP} is transformed into
\begin{equation}\label{rEP1}
\begin{cases}
\begin{aligned}
&(\hat{r}\rho U)_{r}+(\rho V)_{\theta}=0,\\
&\rho UU_{r}+\frac{1}{\hat{r}}\rho VU_{\theta}+\frac{1}{\hat{r}}\rho V^2+P_{r}=\rho\Phi_{r},\\
&\rho UV_{r}-\frac{1}{\hat{r}}\rho UV+\frac{1}{\hat{r}}\rho VV_{\theta}+\frac{1}{\hat{r}}P_{\theta}
=\frac{1}{\hat{r}}\rho\Phi_{\theta},\\
&\rho U\mathscr{K}_r+\frac{1}{\hat{r}} \rho V\mathscr{K}_{\theta}=0,\\
&\Phi_{rr}+\frac{1}{\hat{r}^2}\Phi_{\theta\theta}
-\frac{1}{\hat{r}}\Phi_{r}=\rho-b,
\end{aligned}
\end{cases}
\end{equation}
where
\begin{equation}\label{realK}
\mathscr{K}=\frac{U^2+V^2}{2}+\frac{\gamma \mathrm{e}^{S}\rho^{\gamma-1}}{\gamma-1}-\Phi.
\end{equation}
The two-dimensional convergent nozzle $\tilde\Omega$ becomes
\begin{equation*}
\Omega_{\mathcal{R}}=\Big\{(r,\theta):0<r<\mathcal{R}=r_2-r_1, -\theta_0<\theta<\theta_0\Big\}.
\end{equation*}
\par Before presenting the  stability problem and stating our main result in detail, we first introduce the radially symmetric supersonic solutions to the steady Euler-Poisson system.  Fix $b$ to be a constant $b_0>0$, set $\frac{\mathrm{d}\tilde{\Phi}(\tilde{r})}{\mathrm{d}\tilde{r}}=\tilde{E}(\tilde{r})$. Then
\begin{equation*}
\frac{\mathrm{d}\tilde{\Phi}(\tilde{r})}{\mathrm{d}\tilde{r}}
=-\frac{\mathrm{d}\bar\Phi(r)}{\mathrm{d}r}=\tilde{E}(\tilde{r})=\bar E(r).
\end{equation*}
In the following, we denote $\frac{\mathrm{d}}{\mathrm{d}r}$ by $'$. Suppose that $(\bar\rho, \bar U, \bar P, \bar E)(r)$ solves the ODE system
\begin{equation}\label{-sreq}
\begin{cases}
(\hat{r}\bar \rho \bar U)'=0,\ \ &r\in[0,\mathcal{R}],\\
(\hat{r}\bar \rho \bar U^2)'+\hat{r}\bar P'=-\hat{r}\bar\rho \bar E,\ \ &r\in[0,\mathcal{R}],\\
(\hat{r}\bar\rho\bar U\bar{\mathscr{B}})'=-\hat{r}\bar\rho \bar U\bar E,\ \ &r\in[0,\mathcal{R}],\\
(\hat{r}\bar E)'=-\hat{r}(\bar\rho-b_0),\ \ &r\in[0,\mathcal{R}],\\
\end{cases}
\end{equation}
with the initial value
\begin{equation}\label{-srco}
(\bar\rho,\bar U,\bar P,\bar E)(0)=(\rho_0,U_0,P_0,-E_0),
\end{equation}
where  $\rho_0$, $U_0$, $P_0$ and $E_0$ are positive constants. In addition, the Bernoulli function $\mathscr{B}$ is given by
\begin{equation*}
\bar{\mathscr{B}}=\frac{\bar U^2}{2}+\frac{\gamma \bar{P}}{(\gamma-1){\bar\rho}}=\frac{\bar U^2}{2}+\frac{\gamma e^{\bar S} \bar\rho^{\gamma-1}}{\gamma-1}.
\end{equation*}
It follows from \eqref{-sreq} that
\begin{equation}\label{J-0}
\begin{cases}
\hat{r}(\bar\rho \bar U)(r)=J_0,\quad J_0=r_2\rho_0U_0,\ \ &r\in[0,\mathcal{R}],\\
\bar S(r)=S_0,\quad S_0=\ln\frac{P_0}{\rho_0^{\gamma}},\ \ &r\in[0,\mathcal{R}],\\
\bar{ \mathscr{B}}'=-E, \ \ &r\in[0,\mathcal{R}].\\
\end{cases}
\end{equation}
Furthermore, it holds that
\begin{equation*}
\begin{split}
\bar\rho'=\frac{\bar\rho\bar E}{\bar c^2(\bar M^2-1)}+\frac{\bar\rho \bar M^2}{\hat{r}(\bar M^2-1)},\quad
\bar U'=\frac{\bar U\bar E}{\bar c^2(1-\bar M^2)}+\frac{\bar U}{\hat{r}(1-\bar M^2)},
\end{split}
\end{equation*}
provided $\bar M^2(r)=\frac{\bar U^2}{\bar c^2}$. The definition of $\bar M^2$ gives that
\begin{equation*}
\begin{split}
\bar\rho=\left(\frac{J_0^2}{\gamma\mathrm{e}^{S_0}\hat{r}^2\bar M^2}\right)^{\frac{1}{\gamma+1}}
:=\mu_0\left(\frac{1}{\hat{r}^2\bar M^2}\right)^{\frac{1}{\gamma+1}},\ \
\bar c^2=\gamma\mathrm{e}^{S_0}\bar \rho^{\gamma-1}
=\gamma\mathrm{e}^{S_0}\mu_0^{\gamma-1}
\left(\frac{1}{\hat{r}^2\bar M^2}\right)^{\frac{\gamma-1}{\gamma+1}}.
\end{split}
\end{equation*}
Then direct calculations deduce that
\begin{equation}\label{M2}
\begin{split}
(\bar M^2)'&=\frac{\bar M^2}{1-\bar M^2}\left(\frac{(\gamma+1)\bar E}{\bar c^2}+\frac{2+(\gamma-1)\bar M^2}{\hat{r}}\right)\\
&=\frac{\bar M^2}{1-\bar M^2}
\bigg(\frac{(\gamma+1)\bar E}{\gamma\mathrm{e}^{S_0}\mu_0^{\gamma-1}}(\hat{r}^2\bar M^2)
^{\frac{\gamma-1}{\gamma+1}}+\frac{2+(\gamma-1)\bar M^2}{\hat{r}}\bigg)\\
&:=h_1(r,\bar M^2,\bar E)
\end{split}
\end{equation}
and
\begin{equation}\label{rE}
\begin{split}
(\hat{r}\bar E)'=-\hat{r}\bigg(\mu_0\bigg(\frac{1}{\hat{r}^2\bar M^2}
\bigg)^{\frac{1}{\gamma+1}}-b_0\bigg):=h_2(r,\bar M^2,\bar E).
\end{split}
\end{equation}
This, together with \eqref{J-0}, shows that the initial value problem \eqref{-sreq}-\eqref{-srco} is equivalent to
\begin{equation}\label{sreq1}
\begin{cases}
\hat{r}(\bar\rho \bar U)(r)=J_0, \quad \bar S(r)=S_0,\ \ &r\in[0,\mathcal{R}],\\
(\bar M^2)'=h_1(r,\bar M^2,\bar E),\ \ &r\in[0,\mathcal{R}],\\
(\hat{r}\bar E)'=h_2(r,\bar M^2,\bar E),\ \ &r\in[0,\mathcal{R}],\\
\end{cases}
\end{equation}
with
the the initial value:
\begin{equation}\label{-srco-c}
(\bar M^2,\bar E)(0)=\bigg(\frac{\rho_0U_0^2}{\gamma P_0},E_0\bigg).
\end{equation}
\begin{proposition}\label{pr1}
Fix $\gamma>1 $ and $ b_0>0$. For given positive constants $ r_2 $, $\rho_0$, $U_0$ and $P_0$ satisfying
$$\frac{\rho_0U_0^2}{\gamma P_0}>1, $$  if $r_1\in(0,r_2)$ satisfies
\begin{equation}\label{LE}
\ln\frac{r_2}{r_1}<\frac{\gamma+1}{2(\gamma-1)},
\end{equation}
 then there exists a constant $\check{E}>0$ depending only on the data $(r_1, r_2, \gamma,b_0, \rho_0,U_0,P_0 )$ such that, for any $E_0>\check{E}$, the initial value problem \eqref{sreq1} with
\eqref{-srco-c}
has a unique smooth supersonic solution $(\bar M^2,\bar E)(r)$ satisfying
\begin{equation}
(\bar M^2)'>0\quad\hbox{for}\quad r\in[0,\mathcal{R}],
\end{equation}
provided $\mathcal{R}=r_2-r_1$.
\end{proposition}
The detailed proof of Proposition \ref{pr1} is given in  \cite[Proposition 3.1]{Bae7}, so we omit it.
Define
\begin{equation}\label{back1}
\bar U(r)=\frac{J_0}{\hat{r}\bar \rho},\  \bar P(r)=\mathrm{e}^{S_0}\bar\rho^{\gamma}, \ \
\bar{\Phi}(r)=-\int_{0}^r\bar E(t)\mathrm{d}t+\bigg(\frac{U_0^2}{2}
+\frac{\gamma \mathrm{e}^{S_0}}{\gamma-1}\bigg(\frac{J_0}{r_2U_0}\bigg)^{\gamma-1}\bigg), \
\ r\in[0,\mathcal{R}].
\end{equation}
Then $(\bar\rho,\bar U,\bar P,\bar{\Phi})(r)$ satisfies \eqref{rEP1} in $\Omega_{\mathcal{R}}$. Furthermore,
it follows from the third equation in \eqref{-sreq} and \eqref{back1} that
\begin{equation}\label{backK}
\bar{\mathscr{B}}-\bar{\Phi}\equiv 0,\ \ r\in[0,\mathcal{R}].
\end{equation}

\begin{definition}\label{defe1}
$(\bar\rho,\bar  U, \bar P, \bar \Phi)$   is called the  background  solution  associated with the entrance data $(b_0, \rho_0,$\\$ U_{0},P_{0}, E_{0})$.
\end{definition}
\par For each $ 0<R< \mathcal{R}$, define
\begin{equation*}
\Omega_{{R}}=\Big\{(r,\theta):0<r<R, -\theta_0<\theta<\theta_0\Big\}.
\end{equation*}
The entrance,  exit,  and  nozzle walls of $\Omega_{{R}} $ are  denoted by
\begin{equation*}
\begin{split}
\Gamma_{en}&=\{(r,\theta):r=0,-\theta_0<\theta<\theta_0\},\\
\Gamma_{ex}&=\{(r,\theta):r={R},-\theta_0<\theta<\theta_0\},\\
\Gamma_w&=\{(r,\theta):0<r<{R},\theta=\pm\theta_0\}.
\end{split}
\end{equation*}
The main concern in this paper is to solve the following problem.
\begin{problem}\label{pro1}
For given functions $(U_{en},V_{en}, E_{en}, \mathscr{K}_{en}, S_{en}, \Phi_{ex}, b)$  sufficiently close to $(U_{0},0, E_0, \mathscr{K}_{0},$\\$ S_0, \bar\Phi(R), b_0)$, find a solution $(U,V,\Phi,\mathscr{K},S)$ to the system \eqref{rEP1} in $\Omega_{{R}}$, subject to the boundary conditions
\begin{equation}\label{BC}
\begin{cases}
(U,V,\mathscr{K},S,\Phi_r)(0,\theta)=(U_{en},V_{en},\mathscr{K}_{en},S_{en},
-E_{en})(\theta),\ \ &{\rm{on}}\ \ \Gamma_{en},\\
\Phi(R,\theta)=\Phi_{ex}(\theta),\ \ &{\rm{on}}\ \ \Gamma_{ex},\\
V(r,\pm\theta_0)=\Phi_{\theta}(r,\pm\theta_0)=0, \ \ &{\rm{on}}\ \ \Gamma_w.
\end{cases}
\end{equation}
with $U>0$. Furthermore, the flow is supersonic in $\overline{\Omega_{{R}}}$.
\end{problem}
Here we introduce the following notations. For  given functions
\begin{equation*}
(U_{en},V_{en}, E_{en}, \mathscr{K}_{en}, S_{en}, \Phi_{ex}, b)\in C^{3}([-\theta_0,\theta_0])\times(C^{4}([-\theta_0,\theta_0]))^5\times C^{2}(\overline{\Omega_{R}})
\end{equation*}
with $\mathop{min}\limits_{[-\theta_0,\theta_0]}U_{en}>0$, set
\begin{equation}\label{sigmar}
\begin{split}
\sigma:=\sigma_1(b)+\sigma_2(U_{en}, V_{en}, E_{en}, \Phi_{ex})+\sigma_3(\mathscr{K}_{en}, S_{en}),
\end{split}
\end{equation}
where
\begin{equation*}
\begin{aligned}
\sigma_1(b)&:=
\Vert b-b_0\Vert_{C^{2}(\overline{\Omega_{R}})},\\
\sigma_2(U_{en}, V_{en}, E_{en}, \Phi_{ex})&:=
\Vert U_{en}-U_0\Vert_{C^{3}([-\theta_0,\theta_0])}+\Vert (V_{en}, E_{en}, \Phi_{ex})-(0, E_0, \bar\Phi(R))\Vert_{C^{4}([-\theta_0,\theta_0])},\\
\sigma_3(\mathscr{K}_{en}, S_{en})&:=
\Vert (\mathscr{K}_{en}, S_{en})
-(0, S_0)\Vert_{C^{4}([-\theta_0,\theta_0])}.
\end{aligned}
\end{equation*}
Furthermore, suppose that $(U_{en},V_{en}, E_{en}, \mathscr{K}_{en}, S_{en}, \Phi_{ex}, b)$ satisfies
\begin{equation}\label{compatibility conditions}
\begin{split}
&\partial_{\theta}b=0 \quad\hbox{on}\quad \Gamma_w,\\
&\frac{\mathrm{d} U_{en}}{\mathrm{d}{\theta}}(\pm\theta_0)=\frac{\mathrm{d}^{k-1} V_{en}}{\mathrm{d}{\theta}^{k-1}}(\pm\theta_0)=\frac{\mathrm{d}^{k} E_{en}}{\mathrm{d}{\theta}^{k}}(\pm\theta_0)
=\frac{\mathrm{d}^k \Phi_{ex}}{\mathrm{d}{\theta}^k}(\pm\theta_0)\\
&=\frac{\mathrm{d}^k \mathscr{K}_{en}}{\mathrm{d}{\theta}^k}(\pm\theta_0)=\frac{\mathrm{d}^k S_{en}}{\mathrm{d}{\theta}^k}(\pm\theta_0)=0\quad\hbox{for}\quad k=1,3.
\end{split}
\end{equation}
The main theorem of this paper can be stated as follows.
\begin{theorem}\label{thm1}
Let  $\mathcal{R}$ be given in Proposition \ref{pr1}. Then there exist constants $\bar R\in(0,\mathcal{R})$, $\sigma_{*}>0$, $\mathcal{C}^*_1>0$  and $\kappa_*>0$ depending only on the background data  so that, whenever $R\in(0,\bar R]$, if  $(U_{en},V_{en}, E_{en}, \mathscr{K}_{en}, S_{en}, \Phi_{ex}, b)$ satisfies $\sigma\leq \sigma_{*}$ and the compatibility condition \eqref{compatibility conditions}, then Problem \ref{pro1} has a unique supersonic solution $(U,V,\Phi,\mathscr{K},S)\in [H^3(\Omega_{R})]^2\times[H^4(\Omega_{R})]^3$ satisfying
\begin{equation}\label{sH4}
\Vert(U,V)-(\bar U,0)\Vert_{H^3(\Omega_{R})}
+\Vert\Phi-\bar\Phi\Vert_{H^4(\Omega_{R})}\leq \mathcal{C}^*_1\sigma,
\end{equation}
\begin{equation}\label{sH4-1}
\Vert(\mathscr{K},S)-(0,S_0)\Vert_{H^4(\Omega_{R})}
+\Vert\Phi-\bar\Phi\Vert_{H^4(\Omega_{R})}\leq \mathcal{C}^*_1\sigma_3(\mathscr{K}_{en}, S_{en})
\end{equation}
and
\begin{equation}\label{sc}
c^2(U,V,\Phi,\mathscr{K})-(U^2+V^2)\leq-\kappa_* \quad\hbox{in}\quad \overline{\Omega_{R}},
\end{equation}
provided
\begin{equation*}
c^2(U,V,\Phi,\mathscr{K})=(\gamma-1)\left(\Phi-\frac{U^2+V^2}{2}+\mathscr{K}\right).
\end{equation*}
Furthermore, the solution $(U,V,\Phi,\mathscr{K},S)$ satisfies the compatibility conditions
\begin{equation}\label{cc}
\partial_{\theta}U=\partial^{k-1}_{\theta}V=\partial^k_{\theta}\Phi=\partial^k_{\theta}\mathscr{K}=\partial^k_{\theta}S=0 \quad\hbox{on}\ \ \Gamma_w
\end{equation}
for $k=1,3$ in the sense of trace. Furthermore, for each $\alpha\in(0,1)$, it follows from \eqref{sH4} and \eqref{sH4-1} that
\begin{equation}
\Vert(U,V)-(\bar U,0)\Vert_{C^{1,\alpha}(\overline{\Omega_{R}})}
+\Vert\Phi-\bar\Phi\Vert_{C^{2,\alpha}(\overline{\Omega_{R}})}\leq \mathcal{C}^*_2\sigma
\end{equation}
and
\begin{equation}
\Vert(\mathscr{K},S)-(0,S_0)\Vert_{C^{2,\alpha}(\overline{\Omega_{R}})}
\leq \mathcal{C}^*_2\sigma_3(\mathscr{K}_{en}, S_{en})
\end{equation}
for  some constant $\mathcal{C}^*_2>0$ depending only on the background data and $\alpha$.
\end{theorem}

\section{The linearized problem}\noindent
\par In this section, the deformation-curl-Poisson decomposition is employed to reformulate the steady Euler-Poisson system. Then the corresponding deformation-curl-Poisson system is decomposed into a second order hyperbolic-elliptic system with two transport equations, subject to the nonlinear boundary conditions for Problem \ref{pro1}. A priori estimates are obtained to establish the well-posedness of the associated linearized boundary value problem.
\subsection{The deformation-curl-Poisson decomposition}\noindent
\par If $\rho>0$ in $\overline{\Omega_R}$, it can be directly derived from the steady Euler-Poisson system \eqref{rEP1} that the entropy and the  pseudo-Bernoulli's quantity are transported by the equations
\begin{equation}\label{2-x}
\begin{split}
&US_r+\frac{1}{\hat{r}}VS_{\theta}=0,\\
&U\mathscr{K}_r+\frac{1}{\hat{r}} V\mathscr{K}_{\theta}=0.
\end{split}
\end{equation}
Furthermore, it follows from \eqref{realK} that the density can be represented as
\begin{equation}\label{2rho}
\rho=\mathcal{H}(S,\mathscr{K},U,V,\Phi)=
\bigg(\frac{\gamma-1}{\gamma\mathrm{e}^{S}}
\bigg(\mathscr{K}+\Phi-\frac{U^2+V^2}{2}\bigg)\bigg)
^{\frac{1}{\gamma-1}}
\end{equation}
Substituting \eqref{2rho} into the continuity equation in \eqref{rEP1} leads to
\begin{equation}\label{2rho-de}
\begin{aligned}
&A_{11}(\mathscr{K},U,V,\Phi)U_r+A_{22}(\mathscr{K},U,V,\Phi)V_{\theta}
+A_{12}(U,V)U_{\theta}\\
& +A_{21}(U,V)V_{r}
+B(\mathscr{K},U,V,\Phi)
=0,
\end{aligned}
\end{equation}
where
\begin{equation*}
\begin{split}
&A_{11}(\mathscr{K},U,V,\Phi)=c^2(\mathscr{K},U,V,\Phi)-U^2,\\
&A_{22}(\mathscr{K},U,V,\Phi)=\frac{1}{\hat{r}}(c^2(\mathscr{K},U,V,\Phi)-V^2),\\
&A_{12}(U,V)=-\frac{1}{\hat{r}}UV,\quad
A_{21}(U,V)=-UV,\\
&B(\mathscr{K},U,V,\Phi)=
-\frac{1}{\hat{r}}c^2(\mathscr{K},U,V,\Phi)U+
U\Phi_r+\frac{1}{\hat{r}}V\Phi_{\theta}.\\
\end{split}
\end{equation*}
One can follow from the second equation in \eqref{rEP1} to derive that
\begin{equation}\label{2vor}
U((\hat{r}V)_r-U_{\theta})=\frac{\mathrm{e}^{S}\mathcal{H}^{\gamma-1}(S,\mathscr{K},U,V,\Phi)}
{\gamma-1}S_{\theta}-\mathscr{K}_{\theta}.
\end{equation}
If a smooth flow does not contain the vacuum and the stagnation points, then the system \eqref{rEP1} is equivalent to
\begin{equation}\label{r1EP1}
\begin{cases}
A_{11}(\mathscr{K},U,V,\Phi)U_r+A_{22}(\mathscr{K},U,V,\Phi)V_{\theta}
+A_{12}(U,V)U_{\theta}\\
+A_{21}(U,V)V_{r}
+B(\mathscr{K},U,V,\Phi)
=0,\\
(\hat{r}\Phi_{r})_r+\frac{1}{\hat{r}}\Phi_{\theta\theta}
=\hat{r}(\mathcal{H}(S,\mathscr{K},U,V,\Phi)-b),\\
U((\hat{r}V)_r-U_{\theta})=\frac{\mathrm{e}^{S}\mathcal{H}^{\gamma-1}(S,\mathscr{K},\Phi,U,V)}
{\gamma-1}S_{\theta}-\mathscr{K}_{\theta},\\
US_r+\frac{1}{\hat{r}}VS_{\theta}=0,\\
U\mathscr{K}_r+\frac{1}{\hat{r}} V\mathscr{K}_{\theta}=0.
\end{cases}
\end{equation}
On the other hand, $(\bar U,\bar\Phi)(r)$ satisfies
\begin{equation}\label{sr1EP1}
\begin{cases}
\begin{aligned}
&\bar{A}_{11}\bar{U}'+\bar{B}=0,\\
&(\hat{r}\bar{\Phi}')'
=\hat{r}(\bar{\mathcal{H}}(S_0,0,\bar{U},\bar{\Phi})-b_0),\\
\end{aligned}
\end{cases}
\end{equation}
provided
\begin{equation*}
\begin{split}
&\bar{A}_{11}=A_{11}(0,\bar{U},0,\bar{\Phi})=
\bar c^2-\bar{U}^2,\ \
\bar{B}=B(0,\bar{U},0,\bar{\Phi})=-\frac{1}{\hat{r}}\bar c^2\bar{U}+\bar{U}\bar{\Phi}_r,\\
&  \bar c^2=\bigg((\gamma-1)\bigg(\bar\Phi-\frac{\bar U^2}{2}\bigg)\bigg), \ \ \bar{\mathcal{H}}(S_0,0,\bar{U},\bar{\Phi})=
\left(\frac{\gamma-1}{\gamma\mathrm{e}^{S_0}}\left(\bar{\Phi}-\frac{\bar{U}^2}{2}\right)\right)
^{\frac{1}{\gamma-1}}.
\end{split}
\end{equation*}
\par Define
\begin{equation*}\label{newva}
(\mathcal{U},\mathcal{V},\check{\Phi},\mathcal{S},\mathcal{K})(r,\theta)
=(U,V,\Phi,S,\mathscr{K})(r,\theta)-(\bar U(r),0,\bar\Phi(r),S_0,0).
\end{equation*}
It follows from \eqref{2vor}, \eqref{r1EP1} and \eqref{sr1EP1} that the functions $(\mathcal{U},\mathcal{V},\check{\Phi},\mathcal{S},\mathcal{K})$ satisfy the following system
\begin{equation}\label{elli1}
\begin{cases}
\begin{aligned}
&L_1(\mathcal{U},\mathcal{V},\check{\Phi})=f_1(r,\theta;\mathcal{K},\mathcal{U},\mathcal{V},\check{\Phi}),\\
&L_2(\mathcal{U},\check{\Phi})=f_2(r,\theta;\mathcal{S},\mathcal{K},\mathcal{U},\mathcal{V},\check{\Phi}),\\
&(\hat{r}\mathcal{V})_r-\mathcal{U}_{\theta}=f_3(r,\theta;\mathcal{S},\mathcal{K},\mathcal{U},\mathcal{V},\check{\Phi}),\\
&(\bar U+\mathcal{U})\mathcal{S}_r+\frac{1}{\hat{r}}\mathcal{V}\mathcal{S}_{\theta}=0,\\
&(\bar U+\mathcal{U})\mathcal{K}_r+\frac{1}{\hat{r}}\mathcal{V}\mathcal{K}_{\theta}=0,
\end{aligned}
\end{cases}
\end{equation}
provided
\begin{equation*}
\begin{split}
L_1(\mathcal{U},\mathcal{V},\check{\Phi})=&
\hat{a}_{11}(r,\theta;\mathcal{K},\mathcal{U},\mathcal{V},\check{\Phi})\mathcal{U}_r
+\hat{a}_{22}(r,\theta;\mathcal{K},\mathcal{U},\mathcal{V},\check{\Phi})\mathcal{V}_{\theta}\\
&+\hat{a}_{12}(\mathcal{U},\mathcal{V})\mathcal{U}_{\theta}
+\hat{a}_{21}(\mathcal{U},\mathcal{V})\mathcal{V}_{r}
+\bar{a}_1(r)\mathcal{U}+\bar{b}_1(r)\check{\Phi}+\bar{b}_2(r)\check{\Phi}_r,\\
L_2(\mathcal{U},\check{\Phi})=&(\hat{r}\check{\Phi}_r)_{r}
+\frac{1}{\hat{r}}\check{\Phi}_{\theta\theta}
+\bar{c}_1(r)\mathcal{U}+\bar{c}_2(r)\check{\Phi},
\end{split}
\end{equation*}
and
\begin{equation*}
\begin{aligned}
&\hat{a}_{ii}(r,\theta;\mathcal{K},\mathcal{U},\mathcal{V},\check{\Phi})
=\frac{1}{\bar{c}^2}A_{ii}(\mathcal{K},\bar U+\mathcal{U},\mathcal{V},\bar\Phi+\check{\Phi})\quad {\rm{for}} \quad i=1,2,\\
&\hat{a}_{ij}(r,\theta;\mathcal{U},\mathcal{V})=\frac{1}{\bar{c}^2}A_{ij}(\bar U+\mathcal{U},\mathcal{V})\quad {\rm{for}} \quad i\neq j=1,2,\\
\end{aligned}
\end{equation*}
\begin{equation*}
\begin{aligned}
&\bar{a}_1(r)=\frac{(\gamma-1)\bar U^2-\bar c^2}{\hat{r}\bar c^2}-\frac{1}{\bar{c}^2}\left((\gamma+1)\bar U\bar U'+\bar E\right)
=\frac{\bar U^2}{\bar c^2}\left(\frac{\gamma-1}{\hat{r}}-\frac{\gamma \bar U'}{\bar U}\right)-\frac{\bar U'}{\bar U},\\
&\bar{b}_1(r)=-\frac{(\gamma-1)\bar U}{\hat{r}\bar c^2}+\frac{(\gamma-1)\bar U'}{\bar c^2},\quad \bar{b}_2=\frac{\bar U}{\bar c^2},\ \
\bar{c}_1(r)=\frac{\hat r\bar\rho\bar U}{\bar c^2},\quad \bar{c}_2(r)=-\frac{\hat r\bar\rho}{\bar c^2},\\
&f_1(r,\theta;\mathcal{K},\mathcal{U},\mathcal{V},\check{\Phi})=
-(A_{11}(\mathcal{K},\bar U+\mathcal{U},\mathcal{V},\bar\Phi+\check{\Phi})-\bar A_{11})\frac{\bar U'}{\bar c^2}
-\frac{1}{\bar c^2}B(\mathcal{K},\bar U+\mathcal{U},\mathcal{V},\bar\Phi+\check{\Phi})\\
&+\frac{\bar B}{\bar c^2}+\bar{a}_1(r)\mathcal{U}+\bar{b}_1(r)\check{\Phi}+\bar{b}_2(r)\check{\Phi}_r,\\
&f_2(r,\theta;\mathcal{S},\mathcal{K},\mathcal{U},\mathcal{V},\check{\Phi})
=\hat r\mathcal{H}(S_0+\mathcal{S},\mathcal{K},\bar U+\mathcal{U},\mathcal{V},\bar\Phi+\check{\Phi})-\hat r\bar {\mathcal{H}}+\hat r(b-b_0)
+\bar{c}_1(r)\mathcal{U}+\bar{c}_2(r)\check{\Phi},\\
&f_3(r,\theta;\mathcal{S},\mathcal{K},\mathcal{U},\mathcal{V},\check{\Phi})
=\frac{1}{\bar U+\mathcal{U}}\left(\frac{\mathrm{e}^{S_0+\mathcal{S}}\mathcal{H}^{\gamma-1}(S_0+\mathcal{S},
\mathcal{K},\bar U+\mathcal{U},\mathcal{V},\bar\Phi+\check{\Phi})}
{\gamma-1}\mathcal{S}_{\theta}-\mathcal{K}_{\theta}\right).
\end{aligned}
\end{equation*}
Set $\hat{a}^0_{ii}(r)=\frac{1}{\bar c^2}A_{ii}(0,\bar U,0,\bar\Phi)$ for $i=1,2$, $\hat{a}^0_{ij}(r)=\frac{1}{\bar c^2}A_{ij}(\bar U,0)$ for $i\neq j=1,2$. Then
\begin{equation}\label{coef1}
\begin{split}
\hat a^0_{11}(r)=1-\frac{\bar U^2}{\bar c^2},\quad
\hat a^0_{12}(r)=\hat a^0_{21}(r)=0,\quad
\hat a^0_{22}(r)=\frac{1}{\hat{r}}.
\end{split}
\end{equation}
Furthermore, the boundary conditions for $(\mathcal{U},\mathcal{V},\check{\Phi},\mathcal{S},\mathcal{K})$ are derived as
\begin{equation}\label{aBC1}
\begin{cases}
(\mathcal{U},\mathcal{V},\check{\Phi}_r)(0,\theta)=(U_{en}-U_0,V_{en},-E_{en}+E_0)
(\theta),\ \ &{\rm{on}}\ \ \Gamma_{en},\\
(\mathcal{S},\mathcal{K})(0,\theta)=(S_{en}-S_0,\mathscr{K}_{en})(\theta),\ \ &{\rm{on}}\ \ \Gamma_{en},\\
\check{\Phi}(R,\theta)=\Phi_{ex}(\theta)-\bar\Phi(R),\ \ &{\rm{on}}\ \ \Gamma_{ex},\\
\mathcal{V}(r,\pm\theta_0)=\check{\Phi}_{\theta}(r,\pm\theta_0)=0, \ \ &{\rm{on}}\ \ \Gamma_w.
\end{cases}
\end{equation}
\begin{remark}
{\it Note that $f_3$ in \eqref{elli1} belongs to $H^2(\Omega_R)$ when one looks for the solution $(\mathcal{S},\mathcal{K},\mathcal{U},\mathcal{V},\check{\Phi})$ in $[H^3(\Omega_R)]^4\times H^4(\Omega_R)$. However, due to the structure of the deformation-curl-Poisson system in the supersonic region, the velocity field attains at most the $H^2$ regularity, resulting in a loss of derivatives. To overcome this difficulty, we introduce the stream function via the continuity equation. The stream function has one order higher regularity than the velocity field, and both the entropy and Bernoulli's function can be expressed in terms of it. This approach helps to avoid the possibility of losing derivatives. To this end, we choose suitable function spaces and construct a two-layer iterative scheme.}
\end{remark}

\subsection{The linearized hyperbolic-elliptic system}\noindent

\par We investigate the nonlinear boundary value problem \eqref{elli1} and \eqref{aBC1} by employing the method of iteration. Here we define two iteration sets $\mathcal{I}_{\delta_{\mu}}$ and $\mathcal{I}_{\delta_{\nu}}$ by
\begin{equation}\label{1Set}
\begin{split}
\mathcal{I}_{\delta_{\mu},R}:=\{&(\mu_1,\mu_2)\in [H^4(\Omega_R)]^2: \Vert(\mu_1,\mu_2)\Vert_{H^4(\Omega_R)}<\delta_{\mu},\\ &\partial^k_{\theta}\mu_1=\partial^k_{\theta}\mu_2=0\,\,\hbox{on} \,\,\Gamma_{w}\,\,\hbox{for}\,\, k=1,3\},
\end{split}
\end{equation}
and
\begin{equation}\label{2Set}
\begin{split}
\mathcal{I}_{\delta_{\nu},R}:=\{&(\nu_1,\nu_2,\nu_3)\in [H^3(\Omega_R)]^2\times H^4(\Omega_R): \Vert(\nu_1,\nu_2)\Vert_{H^3(\Omega_R)}+\Vert\nu_3\Vert_{H^4(\Omega_R)}<\delta_{\nu},\\
&\partial_{\theta}\nu_1=\partial^{k-1}_{\theta}\nu_2=\partial^{k}_{\theta}\nu_3=0\,\,\hbox{on} \,\,\Gamma_{w}\,\,\hbox{for}\,\, k=1,3\},
\end{split}
\end{equation}
where the positive constants $\delta_{\mu}$, $\delta_{\nu}$ and $ R $ will be specified later.
\par For fixed $(\mathcal{S}^*,\mathcal{K}^*)\in \mathcal{I}_{\delta_{\mu},R}$ and each $\bm{\mathcal{P}}=(\mathcal{U}^*,\mathcal{V}^*,\check{\Phi}^*)\in \mathcal{I}_{\delta_{\nu},R}$, one can define a linear operator $L^{\A}$ by
\begin{equation*}
\begin{split}
L_1^{\A}(\mathcal{U},\mathcal{V},\check{\Phi})=
\hat{a}_{11}^{\A}\mathcal{U}_r
+\hat{a}_{22}^{\A}\mathcal{V}_{\theta}
+\hat{a}_{12}^{\A}\mathcal{U}_{\theta}
+\hat{a}_{21}^{\A}\mathcal{V}_{r}
+\bar{a}_1\mathcal{U}+\bar{b}_1\check{\Phi}+\bar{b}_2\check{\Phi}_r
\end{split}
\end{equation*}
with
\begin{equation*}
\begin{split}
&\hat{a}_{11}^{\A}=\hat{a}_{11}(r,\theta;\mathcal{K}^*,\mathcal{U}^*,\mathcal{V}^*,\check{\Phi}^*),\\
&\hat{a}_{22}^{\A}=\hat{a}_{22}(r,\theta;\mathcal{K}^*,\mathcal{U}^*,\mathcal{V}^*,\check{\Phi}^*),\\
&\hat{a}_{12}^{\A}=\hat{a}_{12}(r,\theta;\mathcal{U}^*,\mathcal{V}^*),\quad \hat{a}_{21}^{\A}=\hat{a}_{21}(r,\theta;\mathcal{U}^*,\mathcal{V}^*).
\end{split}
\end{equation*}
We first investigate the system for $(\mathcal{U},\mathcal{V},\check{\Phi})$ as follows:
\begin{equation}\label{elli1P}
\begin{cases}
\begin{aligned}
&L_1^{\A}(\mathcal{U},\mathcal{V},\check{\Phi})=f_1(r,\theta;\mathcal{K}^*,\mathcal{U}^*,\mathcal{V}^*,\check{\Phi}^*),\\
&L_2(\mathcal{U},\check{\Phi})=f_2(r,\theta;\mathcal{S}^*,\mathcal{K}^*,\mathcal{U}^*,\mathcal{V}^*,\check{\Phi}^*),\\
&(\hat{r}\mathcal{V})_r-\mathcal{U}_{\theta}=f_3(r,\theta;\mathcal{S}^*,\mathcal{K}^*,\mathcal{U}^*,\mathcal{V}^*,\check{\Phi}^*),\\
\end{aligned}
\end{cases}
\end{equation}
subject to the boundary conditions
\begin{equation}\label{aBC1P}
\begin{cases}
(\mathcal{U},\mathcal{V},\check{\Phi}_r)(0,\theta)=(U_{en}-U_0,V_{en},-E_{en}+E_0)
(\theta),\ \ &{\rm{on}}\ \ \Gamma_{en},\\
\check{\Phi}(R,\theta)=\Phi_{ex}(\theta)-\bar\Phi(R),\ \ &{\rm{on}}\ \ \Gamma_{ex},\\
\mathcal{V}(r,\pm\theta_0)=\check{\Phi}_{\theta}(r,\pm\theta_0)=0, \ \ &{\rm{on}}\ \ \Gamma_w.
\end{cases}
\end{equation}
It can be verified that
\begin{equation}\label{4iter1}
\begin{cases}
\Vert f_1(\cdot,\mathcal{K}^*,\mathcal{U}^*,\mathcal{V}^*,\check{\Phi}^*)\Vert_{H^3(\Omega_R)}\leq C(\delta_{\mu}+\delta_{\nu}^2),\\
\Vert f_2(\cdot,\mathcal{S}^*,\mathcal{K}^*,\mathcal{U}^*,\mathcal{V}^*,\check{\Phi}^*)\Vert_{H^2(\Omega_R)}\leq C(\delta_{\mu}+\delta_{\nu}^2),\\
\Vert f_3(\cdot,\mathcal{S}^*,\mathcal{K}^*,\mathcal{U}^*,\mathcal{V}^*,\check{\Phi}^*)\Vert_{H^3(\Omega_R)}\leq C\delta_{\mu},
\end{cases}
\end{equation}
where the constant $C>0$ depends only on the background data.
\par Next, we  consider the following problem
\begin{equation}\label{extend}
\begin{cases}
(\hat r\phi_{r})_r+\frac{1}{\hat{r}}\phi_{\theta\theta}
=f_3,&{\rm{in}}\ \ \Omega_R,\\
\phi_r(0,\theta)=0, \ \ &{\rm{on}}\ \ \Gamma_{en},\\
\phi_r(R,\theta)=0, \ \ &{\rm{on}}\ \ \Gamma_{ex},\\
\phi(r,\pm \theta_0)=0, \ \ &{\rm{on}}\ \ \Gamma_w.
\end{cases}
\end{equation}
To deal with the singularity near the corner, one can use the standard symmetric extension technique to extend $\phi$ and $f_3$ as
\begin{equation*}
(\phi^e,f_3^e)(r,\theta)=
\begin{cases}
(\phi,f_3)(r,\theta) \quad &\hbox{for} \quad (r,\theta)\in [0,R]\times[-\theta_0,\theta_0],\\
-(\phi,f_3)(r,-\theta) \quad &\hbox{for} \quad (r,\theta)\in [0,R]\times(-3\theta_0,-\theta_0),\\
-(\phi,f_3)(r,2\theta_0-\theta) \quad &\hbox{for} \quad (r,\theta)\in [0,R]\times(\theta_0,3\theta_0).\\
\end{cases}
\end{equation*}
Then $\phi^e$ satisfies
\begin{equation*}
\begin{cases}
(\hat r\phi_{r}^e)_r+\frac{1}{\hat{r}}\phi_{\theta\theta}^e=f_3^e,\\
\phi^e(0,\theta)=\phi^e(R,\theta)=0,\\
\phi^e(r,-3\theta_0)=\phi^e(r,3\theta_0)=0.
\end{cases}
\end{equation*}
The standard elliptic theory in \cite{Gilbarg} implies that $\phi\in H^5(\Omega_R)$ satisfies the estimate
\begin{equation}\label{regu}
\Vert\phi\Vert_{H^5(\Omega_R)}\leq \dot{C}\Vert f_3\Vert_{H^3(\Omega_R)}
\end{equation}
and
\begin{equation}\label{regu11}
\frac{\partial^k\phi}{\partial\theta^k}=0 \quad \rm{for}\quad k=2,4\quad \rm{on}\quad \Gamma_{w}.
\end{equation}
Furthermore,for $\alpha\in(0,1)$, the Sobolev inequality yields that
\begin{equation}\label{regu12}
\Vert\phi\Vert_{C^{3,\alpha}(\overline{\Omega_R})}\leq \ddot{C}\Vert f_3\Vert_{H^3(\Omega_R)},
\end{equation}
provided $\dot{C}$ and $\ddot{C}$ are constants.
Define
\begin{equation*}
\check{\mathcal{U}}=\mathcal{U}+\frac{1}{\hat{r}}\phi_{\theta},\quad
\check{\mathcal{V}}=\mathcal{V}-\phi_r.
\end{equation*}
Then \eqref{elli1P} and \eqref{aBC1P} are transformed into
\begin{equation}\label{elli2}
\begin{cases}
\begin{aligned}
&L_1^{\A}(\check{\mathcal{U}},\check{\mathcal{V}},\check{\Phi})=\mathfrak{f}_1,\\
&L_2(\check{\mathcal{U}},\check{\Phi})=\mathfrak{f}_2,\\
&(\hat{r}\check{\mathcal{V}})_r-\check{\mathcal{U}}_{\theta}=0,\\
\end{aligned}
\end{cases}
\end{equation}
with the boundary conditions
\begin{equation}\label{aBC2}
\begin{cases}
(\check{\mathcal{U}},\check{\mathcal{V}},\check{\Phi}_r)(0,\theta)=
(U_{en}-U_0+\frac{1}{\hat{r}}\phi_{\theta},V_{en},-E_{en}+E_0)(0,\theta),\ \ &{\rm{on}}\ \ \Gamma_{en},\\
\check{\Phi}(R,\theta)=\Phi_{ex}(\theta)-\bar\Phi(R),\ \ &{\rm{on}}\ \ \Gamma_{ex},\\
\check{\mathcal{V}}(r,\pm\theta_0)=\check{\Phi}_{\theta}(r,\pm\theta_0)=0, \ \ &{\rm{on}}\ \ \Gamma_w,
\end{cases}
\end{equation}
where
\begin{equation*}
\begin{split}
\mathfrak{f}_1=&f_1(r,\theta;\mathcal{K}^*,\mathcal{U}^*,\mathcal{V}^*,\check{\Phi}^*)\\
&-a_{21}^{\A}\phi_{rr}
+\frac{1}{\hat{r}}a_{12}^{\A}\phi_{\theta\theta}+\left(\frac{1}{\hat{r}}a_{11}^{\A}-a_{22}^{\A}\right)\phi_{r\theta}
+\left(\frac{1}{\hat{r}^2}a_{11}^{\A}+\frac{1}{\hat{r}}\bar{a}_1\right)\phi_{\theta},\\
\mathfrak{f}_2=&f_2(r,\theta;\mathcal{S}^*,\mathcal{K}^*,\mathcal{U}^*,\mathcal{V}^*,\check{\Phi}^*)+\frac{1}{\hat{r}}\bar{c}_1\phi_{\theta}.
\end{split}
\end{equation*}
\par The third equation in \eqref{elli2} implies that there exists a potential function $\psi$ such that
\begin{equation*}
\psi_{\theta}=\hat{r}\check{\mathcal{V}},\quad \psi_{r}=\check{\mathcal{U}},
\quad \psi(0,-\theta_0)=0.
\end{equation*}
Then \eqref{elli2} and \eqref{aBC2} become
\begin{equation}\label{elli3}
\begin{cases}
\begin{aligned}
&\mathcal{L}_1^{\A}(\psi,\check{\Phi})=\mathfrak{f}_1,\\
&\mathcal{L}_2(\psi,\check{\Phi})=\mathfrak{f}_2,
\end{aligned}
\end{cases}
\end{equation}
and
\begin{equation}\label{aBC3}
\begin{cases}
(\psi_{r},\psi_{\theta},\check{\Phi}_r)(0,\theta)=(U_{en}-U_0+
\frac{1}{r_2}\phi_{\theta},r_2V_{en},-E_{en}+E_0)(0,\theta),\ \ &{\rm{on}}\ \ \Gamma_{en},\\
\check{\Phi}(R,\theta)=\Phi_{ex}(\theta)-\bar\Phi(R),\ \ &{\rm{on}}\ \ \Gamma_{ex},\\
\psi_{\theta}(r,\pm\theta_0)=\check{\Phi}_{\theta}(r,\pm\theta_0)=0, \ \ &{\rm{on}}\ \ \Gamma_w,
\end{cases}
\end{equation}
provided
\begin{equation*}
\begin{split}
&\mathcal{L}_1^{\A}(\psi,\check{\Phi})=
a_{11}^{\A}\psi_{rr}
+a_{22}^{\A}\psi_{\theta\theta}
+2a_{12}^{\A}\psi_{r\theta}
+\bar{a}_1\psi_{r}+ a_2^{\A} \psi_{\theta}+
\bar{b}_1\check{\Phi}+\bar{b}_2\check{\Phi}_r,\\
&\mathcal{L}_2(\psi,\check{\Phi})=(\hat r\check{\Phi}_{r})_r+\frac{1}{\hat{r}}\check{\Phi}_{\theta\theta}
+\bar{c}_1\psi_{r}+\bar{c}_2\check{\Phi},\\
&a_{11}^{\A}=\hat{a}_{11}^{\A},\quad a_{22}^{\A}=\frac{1}{\hat{r}}\hat{a}_{22}^{\A},\quad a_{12}^{\A}=\hat{a}_{12}^{\A},\quad
a_2^{\A}=\frac{1}{\hat{r}}\hat{a}_{12}^{\A}.
\end{split}
\end{equation*}
Set $\bar{a}_{ii}(r)=\hat{a}_{ii}(r,0,0,0,0)$ for $i=1,2$, $\bar{a}_{ij}(r)=\hat{a}_{ij}(r,0,0)$ for $i\neq j=1,2$. Then
\begin{equation}\label{coef11}
\begin{split}
\bar{a}_{11}(r)=1-\frac{\bar U^2}{\bar c^2},\quad
\bar{a}_{12}(r)=\bar{a}_{21}(r)=0,\quad
\bar{a}_{22}(r)=\frac{1}{\hat{r}^2}.
\end{split}
\end{equation}
\begin{lemma}\label{pr2}
Let $\mathcal{R}>0$ be given in Proposition \ref{pr1}.
\begin{enumerate}[ \rm (i)]
\item There exists a positive constant $\mu_0^{\prime}\in(0,1)$ depending only on the background data  such that
\begin{equation*}
-\frac{1}{\mu_0^{\prime}}\leq \bar{a}_{11}(r)\leq -\mu_0^{\prime} \quad \hbox{for} \
 r\in[0,\mathcal{R}].
\end{equation*}
 Furthermore, the coefficients $\bar{a}_{11}$, $\bar{a}_{22}$, $\bar{a}_1$, $\bar{b}_1$ and $\bar{b}_2$ are smooth functions in $[0,\mathcal{R}]$. More precisely, for each $k\in\mathbb{Z}^+$, there exists a constant $\bar C_k>0$ depending only on the background data and $k$ such that
\begin{equation*}
\Vert (\bar{a}_{11},\bar{a}_{22},\bar{a}_1,\bar b_1,\bar b_2)\Vert_{C^k([0,\mathcal{R}])}\leq\bar C_k.
\end{equation*}
\item
For each $\epsilon_0\in(0,\mathcal{R})$, there exists a constant $\delta_0>0$ such that, whenever $R\in(0,\mathcal{R}-\epsilon_0]$ and $\max\{\delta_{\mu},\delta_{\nu}\}\leq \delta_0$,  the following properties hold.
\begin{enumerate}[ \rm (a)]
\item
There exists a positive constant $K_1>0$ such that the coefficients $a^{\A}_{ij}$ $(i,j=1,2)$ satisfy
\begin{equation}\label{coef11-11}
\Vert(a^{\A}_{11}-\bar a_{11},a^{\A}_{12},a^{\A}_{2},a^{\A}_{22}-\bar a_{22})\Vert_{H^3(\Omega_R)}\leq K_1(\delta_{\mu}+\delta_{\nu}).
\end{equation}
Furthermore,
\begin{equation*}
a^{\A}_{12}=a^{\A}_{21}=a^{\A}_2=0 \quad \text{on} \quad \Gamma_{w}.
\end{equation*}
 \item There exist positive constants  $\kappa_0$ and $\kappa_1$ such that
\begin{equation}\label{coef11-22}
\begin{aligned}
-\frac{1}{\kappa_0}\leq a_{11}^{\A}\leq-\kappa_0,\quad
\kappa_1\leq a_{22}^{\A}\leq\frac{1}{\kappa_1}, \quad \hbox{in}\quad \overline{\Omega_R}.
\end{aligned}
\end{equation}
Here the constants  $\kappa_0$, $\kappa_1$ and $K_1$ depend only on the background data and $\epsilon_0$.
\end{enumerate}
\end{enumerate}

\end{lemma}
\subsection{A priori estimates for the linearized system}\noindent

\par  For  each $R\in(0,\mathcal{R}-\epsilon_0]$ and $\delta\in(0,\delta_0]$ with $\delta_0>0$ given by Lemma \ref{pr2}, we investigate the linearized system \eqref{elli3} and \eqref{aBC3}. To simplify the boundary conditions, we define
\begin{equation*}
\begin{split}
\varphi=\psi-\mathfrak{g},\quad \Psi=\check{\Phi}-\Phi^*,\ \
\mathfrak{g}=\int_{-\theta_0}^{\theta}r_2V_{en}(s)\mathrm{d}s,\ \
\Phi^*=(r-R)(-E_{en}(\theta)+E_0)+ (\Phi_{ex}(\theta)-\bar \Phi(R)).
\end{split}
\end{equation*}
Then $(\psi,\check{\Phi})$ solves \eqref{elli3} and \eqref{aBC3} if and only if $(\varphi,\Psi)$ solves the equations
\begin{equation}\label{elli4}
\begin{cases}
\begin{aligned}
&\mathcal{L}_1^{\A}(\varphi,\Psi)=F_1,\\
&\mathcal{L}_2(\varphi,\Psi)=F_2,
\end{aligned}
\end{cases}
\end{equation}
with the boundary conditions
\begin{equation}\label{aBC4}
\begin{cases}
(\varphi_{r},\varphi,\Psi_r)(0,\theta)=(g(\theta),0,0), \ \ &{\rm{on}}\ \ \Gamma_{en},\\
\Psi(R,\theta)=0,\ \ &{\rm{on}}\ \ \Gamma_{ex},\\
\varphi_{\theta}(r,\pm\theta_0)=\Psi_{\theta}(r,\pm\theta_0)=0,\ \ &{\rm{on}} \ \ \Gamma_w,
\end{cases}
\end{equation}
provided
\begin{equation*}
\begin{split}
&F_1=\mathfrak{f}_1-r_2a_{22}^{\A}V_{en}'(\theta)-a_2^{\A}r_2V_{en}(\theta)-\bar b_1\Phi^*-\bar b_2(E_0-E_{en}(\theta)),\\
&F_2=\mathfrak{f}_2-\frac{1}{\hat{r}}\big((R-r)E_{en}''(\theta)
+\Phi_{ex}''(\theta)\big)+(E_0-E_{en}(\theta))-\bar c_2\Phi^*,\\
&g(\theta)=U_{en}-U_0+\frac{1}{r_2}\phi_{\theta}(0,\theta).
\end{split}
\end{equation*}
It follows from the expressions of $\mathfrak{f}_1$ and $\mathfrak{f}_2$ that
\begin{equation}\label{exf}
\Vert \mathfrak{f}_1\Vert_{H^3(\Omega_R)}+\Vert \mathfrak{f}_2\Vert_{H^2(\Omega_R)}\leq C(\delta_{\mu}+\delta_{\nu}^2+\sigma_1(b)),
\end{equation}
provided $C>0$ is a constant depending only on background data. Furthermore, $\mathfrak{f}_1$ and $\mathfrak{f}_2$ satisfy the compatibility conditions
\begin{equation}\label{coexf}
\partial_{\theta}\mathfrak{f}_1=\partial_{\theta}\mathfrak{f}_2=0\quad \rm{on}\quad \Gamma_{w}.
\end{equation}
This, together with \eqref{regu}, \eqref{regu11} and \eqref{regu12}, gives that
\begin{equation}\label{reguF1F2}
\begin{split}
&\Vert F_1\Vert_{H^3(\Omega_R)}+\Vert F_2\Vert_{H^2(\Omega_R)}+\Vert g\Vert_{H^3([-\theta_0,\theta_0])}\\
&\leq \Vert \mathfrak{f}_1\Vert_{H^3(\Omega_R)}+\Vert \mathfrak{f}_2\Vert_{H^2(\Omega_R)}+\sigma_2(U_{en}, V_{en}, E_{en}, \Phi_{ex})\\
&\leq(\delta_{\mu}+\delta_{\nu}^2+\sigma).
\end{split}
\end{equation}
Furthermore, it holds that
\begin{equation}\label{regug}
\begin{split}
g'(\pm\theta_0)=0,\quad \partial_{\theta}F_1=\partial_{\theta}F_2=0\quad \rm{on}\quad \Gamma_{w}.
\end{split}
\end{equation}
\par For fixed $(\mathcal{S}^*,\mathcal{K}^*)\in \mathcal{I}_{\delta_{\mu},R}$ and each $\bm{\mathcal{P}}\in \mathcal{I}_{\delta_{\nu},R}$, once the well-posedness of the boundary value problem \eqref{elli4}-\eqref{aBC4} is established, from which the well-posedness of \eqref{elli3}-\eqref{aBC3} follows.
\begin{proposition}\label{prop1}
Let   $\mathcal{R} $ and  $ \delta_0$  be given in  Proposition \ref{pr1} and  Lemma \ref{pr2}, respectively. For each $\epsilon_0\in(0,\frac{1}{10}\mathcal{R})$, there exist a constant $\bar R\in(0,\mathcal{R}-\epsilon_0]$  such that, for each  $R\in(0,\bar R)$ and $\max\{\delta_{\mu},\delta_{\nu}\}\leq \delta_0$, if $(\varphi,\Psi)\in C^2[\overline{\Omega_R})]^2$ solves the boundary value problem \eqref{elli4}-\eqref{aBC4}, then it satisfies
\begin{equation}\label{H^1}
\Vert(\varphi,\Psi)\Vert_{H^1(\Omega_R)}
\leq C\left(\Vert F_1\Vert_{L^2(\Omega_R)}+\Vert F_2\Vert_{L^2(\Omega_R)}
+\Vert g\Vert_{L^2([-\theta_0,\theta_0])}\right),
\end{equation}
where the constant $C$ and $\bar R$ depend only on the background data and $\epsilon_0$.
\end{proposition}
\begin{proof}
\textbf{Step 1: Weighted energy.}
 For a smooth function $G:=G(r)$ to be determined later, set
\begin{equation*}
I_1(\varphi,\Psi,G):=\iint_{\Omega_R}\mathcal{L}^{\A}_1(\varphi,\Psi)G\varphi_r\mathrm{d}r\mathrm{d}\theta
=I_{11}+I_{12}+I_{13}+I_{14},\
I_2(\varphi,\Psi):=\iint_{\Omega_R}\mathcal{L}_2(\varphi,\Psi)\Psi\mathrm{d}r\mathrm{d}\theta.
\end{equation*}
Integration by parts yields
\begin{equation*}
\begin{aligned}
I_{11}:=&\iint_{\Omega_R}a_{11}^{\A}\varphi_{rr}G(r)\varphi_r\mathrm{d}r\mathrm{d}\theta\\
=&\int_{\Gamma_{ex}}\frac{a_{11}^{\A}G(r)}{2}\varphi^2_{r}\mathrm{d}\theta
-\int_{\Gamma_{en}}\frac{a_{11}^{\A}G(r)}{2}\varphi^2_{r}\mathrm{d}\theta
-\iint_{\Omega_R}\left(\frac{a_{11}^{\A}G(r)}{2}\right)_r\varphi^2_{r}\mathrm{d}r\mathrm{d}\theta,\\
I_{12}:=&\iint_{\Omega_R}2a_{12}^{\A}\varphi_{r\theta}G(r)\varphi_r\mathrm{d}r\mathrm{d}\theta
=-\iint_{\Omega_R}\left(a_{12}^{\A}G(r)\right)_{\theta}\varphi^2_{r}\mathrm{d}r
\mathrm{d}\theta,\\
I_{13}:=&\iint_{\Omega_R}a_{22}^{\A}\varphi_{\theta\theta}G(r)\varphi_r\mathrm{d}r\mathrm{d}\theta\\
=&-\iint_{\Omega_R}(a_{22}^{\A}G(r))_{\theta}\varphi_{r}\varphi_{\theta}\mathrm{d}r\mathrm{d}\theta
-\int_{\Gamma_{ex}}\frac{a_{22}^{\A}G(r)}{2}\varphi^2_{\theta}\mathrm{d}\theta
+\iint_{\Omega_R}\left(\frac{a_{22}^{\A}G(r)}{2}\right)_r\varphi^2_{\theta}
\mathrm{d}r\mathrm{d}\theta,\\
I_{14}:=&\iint_{\Omega_R}\left(\bar a_1 G(r)\varphi_r^2+a_2^{\A}G(r)\varphi_r\varphi_{\theta}+\bar b_1 G(r)\varphi_r\Psi+\bar b_2 G(r)\varphi_r\Psi_r\right)\mathrm{d}r\mathrm{d}\theta.
\end{aligned}
\end{equation*}
On the other hand,
\begin{align*}
I_2(\varphi,\Psi)=&\iint_{\Omega_R}\left((\hat r\Psi_{r})_r+\frac{1}{\hat{r}}\Psi_{\theta\theta}
+\bar c_1\varphi_r+\bar c_2\Psi\right)\Psi\mathrm{d}r\mathrm{d}\theta\\
=&-\iint_{\Omega_R}\hat r\Psi_r^2\mathrm{d}r\mathrm{d}\theta
-\iint_{\Omega_R}\frac{1}{\hat{r}}\Psi_{\theta}^2\mathrm{d}r\mathrm{d}\theta
+\iint_{\Omega_R}\bar c_1\varphi_r\Psi\mathrm{d}r\mathrm{d}\theta
+\iint_{\Omega_R}\bar c_2\Psi^2\mathrm{d}r\mathrm{d}\theta.\\
\end{align*}
Hence
\begin{equation}\label{I_1+I_2}
-(I_1+I_2)
=J_{bc}(\varphi,\Psi,G)+J_c(\varphi,\Psi,G)+J_{p}(\varphi,\Psi,G),
\end{equation}
where
\begin{equation*}
\begin{split}
J_{bd}&:=-\int_{\Gamma_{ex}}\frac{a_{11}^{\A}G}{2}\varphi^2_{r}\mathrm{d}\theta
+\int_{\Gamma_{en}}\frac{a_{11}^{\A}G}{2}\varphi^2_{r}\mathrm{d}\theta
+\int_{\Gamma_{ex}}\frac{a_{22}^{\A}G}{2}\varphi^2_{\theta}\mathrm{d}\theta,\\
J_c&:=\iint_{\Omega_R}\big(\partial_{\theta}a_{22}^{\A}-a_2^{\A}\big)G(r)\varphi_{r}\varphi_{\theta}\mathrm{d}r\mathrm{d}\theta
-\iint_{\Omega_R}\bar b_1G(r)\varphi_{r}\Psi\mathrm{d}r\mathrm{d}\theta
-\iint_{\Omega_R}\bar c_1\varphi_r\Psi\mathrm{d}r\mathrm{d}\theta\\
&\quad\ \ -\iint_{\Omega_R}\bar b_2G(r)\varphi_{r}\Psi_r\mathrm{d}r\mathrm{d}\theta,\\
J_p&:=\iint_{\Omega_R}q_1(G,r,\theta)\varphi^2_{r}\mathrm{d}r\mathrm{d}\theta
+\iint_{\Omega_R}q_2(G,r,\theta)\varphi^2_{\theta}\mathrm{d}r\mathrm{d}\theta
+\iint_{\Omega_R}h(r)\Psi^2\mathrm{d}r\mathrm{d}\theta\\
&\quad\ \ +\iint_{\Omega_R}\left(\Psi^2_r+\frac{1}{\hat{r}^2}\Psi_{\theta}^2\right)\mathrm{d}r\mathrm{d}\theta,
\end{split}
\end{equation*}
and
\begin{equation*}
\begin{split}
&q_1(G,r,\theta):=\bigg(\frac{a_{11}^{\A}G(r)}{2}\bigg)_r+\big(a_{12}^{\A}G(r)\big)_{\theta}-\bar a_1G(r),\\
&q_2(G,r,\theta):=-\bigg(\frac{a_{22}^{\A}G(r)}{2}\bigg)_r,\quad h(r):=-\bar c_2.
\end{split}
\end{equation*}
\textbf{Step 2: A priori estimate for $(\varphi,\Psi)$.}
Note that
\begin{equation}
h(r)=\frac{\hat r\bar\rho}{\bar c^2}>0.
\end{equation}
Then for each $R\in(0,\mathcal{R}-\epsilon_0]$,
\begin{equation*}
\mu_{R}:=\inf_{r\in[0,R]} h(r)
\end{equation*}
is a strictly positive constant depending only on background data. It follows from the Cauchy-Schwarz inequality and Lemma \ref{pr2} that
\begin{equation*}
\begin{split}
\vert J_c(\varphi,\Psi,G)\vert
&\leq\iint_{\Omega_R}\bigg(\frac{\hat r}{8}\Psi_r^2+\frac{\mu_{R}}{4}\Psi^2
+q_3(G,r,\theta)\varphi_r^2
+K_*\delta\varphi_{\theta}^2\bigg)\mathrm{d}r\mathrm{d}\theta,
\end{split}
\end{equation*}
provided the constant $K_*>0$ and
\begin{equation*}
q_3(G,r):=2\bigg(\Big(\frac{\bar{b}_2^2}{\hat r}+\frac{\bar{b}_1^2}{\mu_{R}}
+\frac{K_*\delta}{8}\Big)G^2+\frac{\bar{c}_1^2}{\mu_{R}}\bigg).
\end{equation*}
Employing \eqref{coef11-22} and \eqref{I_1+I_2} yields
\begin{equation}\label{I}
\begin{split}
-(I_1+I_2)\geq&
\iint_{\Omega_R}\left(q_1(G,r,\theta)-q_3(G,r)\right)\varphi^2_{r}\mathrm{d}r\mathrm{d}\theta\\
&+\iint_{\Omega_R}\left(q_2(G,r,\theta)-K_*\delta\right)\varphi^2_{\theta}\mathrm{d}r\mathrm{d}\theta\\
&+\iint_{\Omega_R}\frac{3\mu_{R}}{4}\Psi^2\mathrm{d}r\mathrm{d}\theta
+\iint_{\Omega_R}\left(\frac{7\hat r}{8}\Psi^2_r+\frac{1}{\hat{r}}\Psi_{\theta}^2\right)\mathrm{d}r\mathrm{d}\theta\\
&+\min\{\kappa_0,\kappa_1\}
\int_{\Gamma_{ex}}\frac{G}{2}\left(\varphi^2_{r}+\varphi^2_{\theta}\right)\mathrm{d}
\theta
-\int_{\Gamma_{en}}\frac{G}{2\kappa_0}g^2\mathrm{d}\theta,
\end{split}
\end{equation}
provided that the function $G(R)\geq 0$. In order to derive the $H^1$ estimate, $(G,R,\delta)$ should be chosen such that
\begin{equation}\label{c1}
G(r)>0 \quad \hbox{for} \quad 0\leq r\leq R,
\end{equation}
\begin{equation}\label{c2}
-\frac{2}{a_{11}^{\A}}\left(q_1(G,r,\theta)-q_3(G,r)\right)\geq\lambda_0, \quad \hbox{in} \quad \overline{\Omega_R},
\end{equation}
\begin{equation}\label{c3}
\frac{2}{a_{22}^{\A}}\left(q_2(G,r,\theta)-K_*\delta\right)\geq\lambda_0, \quad \hbox{in} \quad \overline{\Omega_R}
\end{equation}
for some constant $\lambda_0>0$. With the aid of \eqref{coef11-22} and \eqref{I}, we deduce that
\begin{equation}\label{II}
\begin{split}
-(I_1+I_2)\geq&
\frac{\lambda_0}{2}\iint_{\Omega_R}\left(\kappa_0\varphi^2_{r}
+\kappa_1\varphi^2_{\theta}\right)\mathrm{d}r\mathrm{d}\theta
+\frac{3\mu_{R}}{4}\iint_{\Omega_R}\Psi^2\mathrm{d}r\mathrm{d}\theta\\
&+\iint_{\Omega_R}\left(\frac{7\hat r}{8}\Psi^2_r+\frac{1}{\hat{r}}\Psi_{\theta}^2\right)\mathrm{d}r\mathrm{d}\theta
-\int_{\Gamma_{en}}\frac{G}{2\kappa_0}g^2\mathrm{d}\theta.
\end{split}
\end{equation}
Since the functions $\varphi$ and $\Psi$ satisfy the problem \eqref{elli4}-\eqref{aBC4}, then we obtain
\begin{equation*}
\begin{split}
-(I_1+I_2)&=-\iint_{\Omega_R}\bigg(G(r)\varphi_rF_1
+\Psi F_2\bigg)\mathrm{d}r\mathrm{d}\theta\\
&\leq\iint_{\Omega_R}\bigg(\frac{\lambda_0\kappa_0}{4}\varphi^2_r
+\frac{1}{\lambda_0\kappa_0}\vert GF_1\vert^2
+\frac{\mu_{R}}{4}\Psi^2+\frac{1}{\mu_{R}}\vert F_2\vert^2\bigg)\mathrm{d}r\mathrm{d}\theta,
\end{split}
\end{equation*}
here the Cauchy-Schwarz inequality is  utilized.
This, together with \eqref{II}, yields
\begin{equation*}
\begin{split}
&\min\{\kappa_0,\kappa_1\}\frac{\lambda_0}{4}\iint_{\Omega_R}\left(\varphi^2_{r}
+\varphi^2_{\theta}\right)\mathrm{d}r\mathrm{d}\theta
+\frac{\mu_{R}}{2}\iint_{\Omega_R}\Psi^2\mathrm{d}r\mathrm{d}\theta
+\iint_{\Omega_R}\left(\frac{7\hat r}{8}\Psi^2_r+\frac{1}{\hat{r}}\Psi_{\theta}^2\right)
\mathrm{d}r\mathrm{d}\theta\\
&\leq\frac{1}{\lambda_0\kappa_0}\iint_{\Omega_R}\vert GF_1\vert^2\mathrm{d}r\mathrm{d}\theta
+\frac{1}{\mu_{R}}\iint_{\Omega_R}\vert F_2\vert^2\mathrm{d}r\mathrm{d}\theta
+\int_{\Gamma_{en}}\frac{G}{2\kappa_0}g^2\mathrm{d}\theta.
\end{split}
\end{equation*}
Consequently, the estimate \eqref{H^1} follows from the boundary condition $\varphi=0$ on $\Gamma_{en}$ and the Poincar$\acute{\mathrm{e}}$ inequality. \\
\textbf{Step 3: Construction of $G$.}
Set
\begin{equation*}
\mu_{R_*}=\inf_{r\in[0,R_*]} h(r) \quad \hbox{for} \quad R_*=\mathcal{R}-\epsilon_0.
\end{equation*}
Therefore,
\begin{equation}\label{mu_1}
\mu_{R}\geq\mu_{R_*}>0 \quad \hbox{for } \quad R\in(0,R_*].
\end{equation}
By \eqref{coef11-22} and \eqref{mu_1}, for any $\delta\in(0,\delta_0]$, $R\in(0,\mathcal{R}-\epsilon_0]$, it holds that
\begin{equation*}
\begin{split}
-\frac{2}{a_{11}^{\A}}\bigg(q_1(G,r,\theta)-q_3(G,r)\bigg)
\geq&
-G_r-2\cdot\bigg(\frac{\partial_ra_{11}^{\A}}{2a_{11}^{\A}}
+\frac{\partial_{\theta}a_{12}^{\A}}{a_{11}^{\A}}-\frac{\bar{a}_1}{a_{11}^{\A}}\bigg)G\\
&-\frac{4}{\kappa_0}\bigg(\frac{\bar{b}_2^2}{\hat r}+\frac{\bar{b}_1^2}{\mu_{R_*}}+
\frac{K_*\delta}{8}\bigg)G^2
-\frac{4\bar{c}_1^2}{\kappa_0\mu_{R_*}}
\end{split}
\end{equation*}
and
\begin{equation*}
\frac{2}{a_{22}^{\A}}\left(q_2(G,r,\theta)-K_*\delta\right)\geq
-G_r-2\cdot\frac{\partial_ra_{22}^{\A}}{2a_{22}^{\A}}G-\frac{2K_*\delta}{\kappa_1}.
\end{equation*}
Define
\begin{equation*}
\begin{split}
\begin{aligned}
\mathfrak{h}^{\epsilon_0}:=&\frac{2K_*\delta_1}{\kappa_1}
+\max_{(0,R_*]}\frac{4\bar{c}_1^2}{\kappa_0\mu_{R_*}},\quad
\mathfrak{q}^{\epsilon_0}:=\max_{(0,R_*]}\frac{4}{\kappa_0}
\left(\frac{\bar{b}_2^2}{\hat r}+\frac{\bar{b}_1^2}{\mu_{R_*}}+\frac{K_*\delta_1}{8}\right),\\
\mathfrak{p}_{1}^{\epsilon_0}:=&\min_{(0,R_*]}
\left(-\frac{\partial_r\bar{a}_{11}}{2\bar{a}_{11}}+\frac{\bar{a}_1}{\bar{a}_{11}}\right),\quad
\mathfrak{p}_{2}^{\epsilon_0}:=\min_{(0,R_*]}
\bigg(-\frac{\partial_r\bar a_{22}}{2\bar a_{22}}\bigg),\quad
\bar{\mathfrak{p}}^{\epsilon_0}:=\min\{\mathfrak{p}_{1}^{\epsilon_0},\mathfrak{p}_{2}^{\epsilon_0}\}.
\end{aligned}
\end{split}
\end{equation*}
It follows from Proposition  \ref{pr1} that there exists a constant $C>0$ such that
\begin{equation*}
0<\mathfrak{h}^{\epsilon_0},\mathfrak{q}^{\epsilon_0}\leq C \quad \hbox{and} \quad \vert \mathfrak{p}^{\epsilon_0}\vert\leq C.
\end{equation*}
For fixed $(\mathcal{S}^*,\mathcal{K}^*)\in \mathcal{I}_{\delta_{\mu},R}$ and each $\bm{\mathcal{P}}=(\mathcal{U}^*,\mathcal{V}^*,\check{\Phi}^*)\in \mathcal{I}_{\delta_{\nu},R}$, set
\begin{equation*}
\begin{split}
\mathfrak{p}_{1}^{\A}:=\min_{(0,R_*]}
\bigg(-\frac{\partial_ra_{11}^{\A}}{2a_{11}^{\A}}-
\frac{\partial_{\theta}a_{12}^{\A}}{a_{11}^{\A}}+\frac{\bar{a}_1}{a_{11}^{\A}}\bigg),\quad
\mathfrak{p}_{2}^{\A}:=\min_{(0,R_*]}
\bigg(-\frac{\partial_ra_{22}^{\A}}{2a_{22}^{\A}}\bigg),\quad
\bar{\mathfrak{p}}^{\A}:=\min\{\mathfrak{p}_{1}^{\A},\mathfrak{p}_{2}^{\A}\},
\end{split}
\end{equation*}
provided $\max\{\delta_{\mu},\delta_{\nu}\}\leq \delta^*$, $R\in(0,R_*]$.
Lemma \ref{pr2} and the Morrey's inequality give
\begin{equation*}
\mathfrak{p}^{\A}\geq \bar{\mathfrak{p}}^{\epsilon_0}-\frac{\tilde C\delta_1}{\kappa_0}:=\mathfrak{p}^{\epsilon_0}
\end{equation*}
for a constant $\tilde C>0$. The conditions \eqref{c1}-\eqref{c3} are satisfied as long as $G>0$ and solves the ODE
\begin{equation}\label{ODE}
-G_r+2\mathfrak{p}^{\epsilon_0}G-\mathfrak{q}^{\epsilon_0}G^2-\mathfrak{h}^{\epsilon_0}=\lambda_0 \quad \hbox{for} \quad 0<r<R.
\end{equation}
The explicit construction of $G$ is given in Proposition 2.4 of \cite{Bae1}. For convenience, we present the solution to this ODE directly.
If $\mathfrak{h}^{\epsilon_0}\mathfrak{q}^{\epsilon_0}-(\mathfrak{p}^{\epsilon_0})^2>0$, $G$ is given by
\begin{equation*}
G(r;\lambda_0)=\chi(\lambda_0)\tan(\frac{\pi}{2}-\lambda_0-
\mathfrak{q}^{\epsilon_0}\chi(\lambda_0)r)+\frac{\mathfrak{p}^{\epsilon_0}}
{\mathfrak{q}^{\epsilon_0}}.
\end{equation*}
Set
\begin{equation*}
\bar R:=\min\left\{R_*,\sup_{\lambda_0\in(0,\lambda_0^*]}\frac{1}{\mathfrak{q}^{\epsilon_0}\chi(\lambda_0)}\left(\frac{\pi}{2}
-\lambda_0-\arctan\left(\frac{\vert \mathfrak{p}^{\epsilon_0}\vert}{\mathfrak{q}^{\epsilon_0}\chi(\lambda_0)}\right)
\right)\right\},
\end{equation*}
where $\lambda_0^*>0$ ia a small constant. For any $R\in(0,\bar R)$, there exists a constant $\lambda_0\in(0,\lambda_0^*]$ such that the conditions \eqref{c1}-\eqref{c3} hold.
If $\mathfrak{h}^{\epsilon_0}\mathfrak{q}^{\epsilon_0}-(\mathfrak{p}^{\epsilon_0})^2<0$, $G$ is given by
\begin{equation*}
G(r;\lambda_0)=\xi(\lambda_0)^{\frac{1}{2}}\cot\left(\xi(\lambda_0)
+\mathfrak{q}^{\epsilon_0}\xi(\lambda_0)^{\frac{1}{2}}r\right)+\frac{\mathfrak{p}^{\epsilon_0}}{\mathfrak{q}^{\epsilon_0}},
\end{equation*}
where
\begin{equation*}
\xi(\lambda_0):=\frac{{\lambda_0-z_*}}{\mathfrak{q}^{\epsilon_0}},\quad z_*=\frac{(\mathfrak{p}^{\epsilon_0})^2-\mathfrak{h}^{\epsilon_0}\mathfrak{q}^{\epsilon_0}}{\mathfrak{q}^{\epsilon_0}}\geq0
\end{equation*}
and $\lambda_0>z_*$ is chosen such that $\xi(\lambda_0)>0$. $G(r;\lambda_0)>0$ for all $r\in[0,R]$ as long as
\begin{equation*}
\frac{1}{R}+\mathfrak{p}^{\epsilon_0}>0 \quad \hbox{and} \quad  R<R_*.
\end{equation*}
Therefore, $G(r)$ depends only on the background data and $\epsilon_0$. This completes the proof of Proposition \ref{prop1}.
\end{proof}

\begin{proposition}\label{prop2}
Fix $\epsilon_0\in(0,\mathcal{R})$. Let $\bar R\in(0,\mathcal{R}-\epsilon_0]$ be given in Proposition \ref{prop1}. Then for $R\in(0,\bar R)$ and $\max\{\delta_{\mu},\delta_{\nu}\}\leq \delta_0$, there exists a constant $C$ depending only on the background data and $\epsilon_0$ such that the boundary value problem \eqref{elli4}-\eqref{aBC4} associated with $(\mathcal{S}^*,\mathcal{K}^*)\in \mathcal{I}_{\delta_{\mu},R}$ and $(\mathcal{U}^*,\mathcal{V}^*,\check{\Phi}^*)\in \mathcal{I}_{\delta_{\nu},R}$ has a unique solution $(\varphi,\Psi)\in [H^4(\Omega_R)]^2$ satisfying
\begin{equation}\label{H^4_1}
\Vert\varphi\Vert_{H^4(\Omega_R)}
\leq C\left(\Vert F_1\Vert_{H^3(\Omega_R)}+\Vert F_2\Vert_{H^2(\Omega_R)}
+\Vert g\Vert_{H^3([-\theta_0,\theta_0])}\right)
\end{equation}
and
\begin{equation}\label{H^4_2}
\Vert\Psi\Vert_{H^4(\Omega_R)}
\leq C\left(\Vert F_1\Vert_{H^2(\Omega_R)}+\Vert F_2\Vert_{H^2(\Omega_R)}
+\Vert g\Vert_{H^2([-\theta_0,\theta_0])}\right).
\end{equation}
Furthermore, the solution $(\varphi,\Psi)$ satisfies the compatibility conditions
\begin{equation*}
\frac{\partial^k\varphi}{\partial\theta^k}=\frac{\partial^k\Psi}{\partial\theta^k}=0 \quad \hbox{on} \quad \Gamma_{w} \quad \hbox{for} \quad k=1,3.
\end{equation*}
\end{proposition}
\begin{proof}
\textbf{Step 1: Approximated linear boundary value problem.}
It follows from Lemma \ref{pr2}, \eqref{reguF1F2} and \eqref{regug} that $(a_{11}^{\A},a_{12}^{\A},a_{22}^{\A},a_{2}^{\A},F_1,F_2)\in\big(H^3(\Omega_R)\big)^5\times H^2(\Omega_R)$ satisfying
\begin{equation*}
a_{12}^{\A}=a_{2}^{\A}=\partial_{\theta}a_{11}^{\A}=\partial_{\theta}a_{22}^{\A}=\partial_{\theta}F_1=\partial_{\theta}F_2=0\quad\rm{on}\quad\Gamma_{w}.
\end{equation*}
Then $(a_{11}^{\A},a_{12}^{\A},a_{22}^{\A},a_{2}^{\A},F_1,F_2)$ can be extended into $\Omega_R^{\mathfrak{e}}:=(0,R)\times(-3\theta_0,3\theta_0)$ by
\begin{equation*}
\begin{aligned}
&(a_{ii}^{\mathfrak{e}},a_{12}^{\mathfrak{e}},a_{2}^{\mathfrak{e}},F_i^{\mathfrak{e}})(r,\theta)=
&\begin{cases}
(a_{ii}^{\A},a_{12}^{\A},a_{2}^{\A},F_i)(r,\theta) \quad &\hbox{for}\quad (r,\theta)\in(0,R) \times [-\theta_0, \theta_0], \\
(a_{ii}^{\A},-a_{12}^{\A},-a_{2}^{\A},F_i)(r,-2\theta_0-\theta) \quad &\hbox{for}\quad (r,\theta)\in(0,R) \times [-3\theta_0, \theta_0], \\
(a_{ii}^{\A},-a_{12}^{\A},-a_{2}^{\A},F_i)(r,2\theta_0-\theta) \quad &\hbox{for}\quad (r,\theta)\in(0,R) \times [\theta_0, 3\theta_0].
\end{cases}
\end{aligned}
\end{equation*}
with $i=1,2$.
One can obtain that $(a_{11}^{\mathfrak{e}},a_{12}^{\mathfrak{e}},a_{22}^{\mathfrak{e}},a_{2}^{\mathfrak{e}},F_1^{\mathfrak{e}},F_2^{\mathfrak{e}})
\in\big(H^3(\Omega_R^{\mathfrak{e}})\big)^5\times H^2(\Omega_R^{\mathfrak{e}})$. Let $\tau_\varrho$ is a radially symmetric standard mollifier with a compact support in a disk of radius $\varrho>0$. For $(r,\theta)\in\Omega_R^{\mathfrak{e}}$, define
\begin{equation*}
a_{ij}^\varrho=a_{ij}^{\mathfrak{e}}\ast\tau_\varrho,\quad a_{2}^\varrho=a_{2}^{\mathfrak{e}}\ast\tau_\varrho,\quad F_i^\varrho=F_i^{\mathfrak{e}}\ast\tau_\varrho\quad \hbox{for}\quad i,j=1,2
\quad \hbox{and}\quad i\leq j.
\end{equation*}
Hence $(a_{11}^\varrho,a_{12}^\varrho,a_{22}^\varrho,a_{2}^\varrho, F_1^\varrho,F_2^\varrho)\subset C^{\infty}(\overline{\Omega_R})$ and
\begin{equation*}
a_{12}^\varrho=a_{2}^\varrho=\partial_{\theta}a_{11}^\varrho=\partial_{\theta}a_{22}^\varrho
=\partial_{\theta}F_1^\varrho=\partial_{\theta}F_2^\varrho=0
\quad\hbox{on}\quad \Gamma_{w}.
\end{equation*}
Furthermore, as $\varrho\to 0$, we have
\begin{equation*}
\begin{split}
&\Vert a_{ij}^\varrho-a_{ij}^{\A}\Vert_{H^3(\Omega_R)}\to 0,\quad \Vert a_{2}^\varrho-a_{2}^{\A}\Vert_{H^3(\Omega_R)}\to 0,\\
&\Vert F_1^\varrho-F_{1}\Vert_{H^3(\Omega_R)}\to 0,\quad
\Vert F_2^\varrho-F_{2}\Vert_{H^2(\Omega_R)}\to 0
\quad\hbox{for}\quad i,j=1,2
\quad \hbox{and}\quad i\leq j.
\end{split}
\end{equation*}
Then a approximation problem with smooth coefficients of \eqref{elli4}-\eqref{aBC4} is obtained as follows:
\begin{equation}\label{Leqs3}
\begin{cases}
\mathcal{L}^\varrho_1(\varphi,\Psi)=F_1^\varrho,\\
\mathcal{L}_2(\varphi,\Psi)=F_2^\varrho,
\end{cases}
\end{equation}
with the boundary conditions
\begin{equation}\label{Lbcs3}
\begin{cases}
(\varphi_{r},\varphi,\Psi_r)(0,\theta)=(g(\theta),0,0), \ \ &{\rm{on}}\ \ \Gamma_{en},\\
\Psi(R,\theta)=0,\ \ &{\rm{on}}\ \ \Gamma_{ex},\\
\varphi_{\theta}(r,\pm\theta_0)=\Psi_{\theta}(r,\pm\theta_0)=0, \ \ &{\rm{on}} \ \ \Gamma_w,
\end{cases}
\end{equation}
provided
\begin{equation*}
\begin{split}
\mathcal{L}^\varrho_1(\varphi,\Psi)=&
a_{11}^\varrho\varphi_{rr}
+a_{22}^\varrho\varphi_{\theta\theta}
+2a_{12}^\varrho\varphi_{r\theta}
+\bar{a}_1\varphi_{r}+ a_2^\varrho \varphi_{\theta}+
\bar{b}_1\Psi+\bar{b}_2\Psi_r,\\
\mathcal{L}_2(\varphi,\Psi)=&(\hat r\Psi_{r})_r+\frac{1}{\hat{r}}\Psi_{\theta\theta}
+\bar{c}_1\varphi_{r}+\bar{c}_2\Psi.
\end{split}
\end{equation*}
\textbf{Step 2: Galerkin approximation.}
For the open interval $\Gamma:=(-\theta_0,\theta_0)$, we define a standard inner product $\langle\cdot,\cdot\rangle$ in $L^2(\Gamma)$ as
\begin{equation*}
\langle\zeta_1,\zeta_2\rangle=\int_{-\theta_0}^{\theta_0}\zeta_1(\theta)\zeta_2(\theta)\mathrm{d}\theta
\end{equation*}
and a function $\eta_0(\theta):=\left(\frac{1}{2\theta_0}\right)^{\frac{1}{2}}$,
$\eta_k(\theta):=\left(\frac{1}{\theta_0}\right)^{\frac{1}{2}}\cos\left(\frac{k\pi}{\theta_0}\theta\right)$
for each $k\in\mathbb{Z}^+$. Then, for the eigenvalue problem
\begin{equation*}
\begin{cases}
-\eta^{\prime\prime}(\theta)=\upsilon\eta(\theta),\\
\eta^{\prime}(\pm\theta_0)=0,
\end{cases}
\end{equation*}
the set $\mathcal{A}:=\{\eta_k\}_{k=0}^{\infty}$ and $\mathcal{A}_1:=\{\upsilon_k:\upsilon_k=\left(\frac{k\pi}{\theta_0}\right)^2,k\in\mathbb{Z}^+\}$ represent the collections of all eigenfunctions and the corresponding eigenvalues, respectively. Furthermore, the set $\mathcal{A}$ forms an orthonormal basis in $L^2(\Gamma)$ and an orthogonal basis in $H^1(\Gamma)$.
\par For any fixed $\varrho>0$, we consider the linear boundary value problem \eqref{Leqs3}-\eqref{Lbcs3}. For each $m=1,2,\cdot\cdot\cdot$ and all $k=0,1,...,m$, let
\begin{equation}\label{m}
\varphi^{\varrho}_m(r,\theta)=\sum_{j=0}^{m}\vartheta^{\varrho}_j(r)\eta_j(\theta),\quad \Psi^{\varrho}_m(r,\theta)=\sum_{j=0}^{m}\nu^{\varrho}_j(r)\eta_j(\theta)
\quad \hbox{for} \quad (r,\theta)\in\Omega_R.
\end{equation}
Then we find a solution $(\varphi^{\varrho}_m,\Psi^{\varrho}_m)$ satisfying
\begin{equation}\label{Leqs4}
\begin{split}
&\langle\mathcal{L}^\varrho_1(\varphi^{\varrho}_m,\Psi^{\varrho}_m),\eta_k\rangle=\langle F_1^\varrho,\eta_k\rangle,\\
&\langle\mathcal{L}_2(\varphi^{\varrho}_m,\Psi^{\varrho}_m),\eta_k\rangle=\langle F_2^\varrho,\eta_k\rangle,
\end{split}
\end{equation}
subject to the boundary conditions
\begin{equation}\label{Lbcs4}
\begin{cases}
(\partial_{r}\varphi^{\varrho}_{m},\varphi^{\varrho}_m,\partial_{r}\Psi^{\varrho}_m)=(\sum_{j=0}^{m}\langle g,\eta_j\rangle\eta_j,0,0) \ \ &{\rm{on}}\ \ \Gamma_{en},\\
\Psi^{\varrho}_m=0\ \ &{\rm{on}}\ \ \Gamma_{ex},\\
\partial_{\theta}\varphi^{\varrho}_m=\partial_{\theta}\Psi^{\varrho}_{m}=0 \ \ &{\rm{on}} \ \ \Gamma_w.
\end{cases}
\end{equation}
If $\varphi^{\varrho}_m$ and $\Psi^{\varrho}_m$ are smooth functions and solve \eqref{Leqs4}-\eqref{Lbcs4}, then
\begin{equation*}
\begin{split}
&-\iint_{\Omega_R}\mathcal{L}^{\varrho}_1(\varphi^{\varrho}_m,\Psi^{\varrho}_m)G(r)\partial_r\varphi^{\varrho}_m\mathrm{d}r\mathrm{d}\theta
-\iint_{\Omega_R}\mathcal{L}_2(\varphi^{\varrho}_m,\Psi^{\varrho}_m)\Psi^{\varrho}_m
\mathrm{d}r\mathrm{d}\theta\\
&=-\int_{0}^{R}\sum_{j=0}^{m}G(r)(\vartheta^{\varrho}_j)'(r)
\langle\mathcal{L}^{\varrho}_1(\varphi^{\varrho}_m,\Psi^{\varrho}_m),\eta_j\rangle\mathrm{d}r
-\int_{0}^{R}\sum_{j=0}^{m}\nu^{\varrho}_j(r)\langle\mathcal{L}_2(\varphi^{\varrho}_m,\Psi^{\varrho}_m),\eta_j\rangle\mathrm{d}r\\
&=-\iint_{\Omega_R}G(r)F^{\varrho}_1\partial_r\varphi^{\varrho}_m\mathrm{d}r\mathrm{d}\theta
-\iint_{\Omega_R}F^{\varrho}_2\Psi^{\varrho}_m\mathrm{d}r\mathrm{d}\theta,
\end{split}
\end{equation*}
provided $G(r)$ is constructed in Proposition \ref{prop1}. Then there holds
\begin{equation}\label{H1}
\Vert\varphi^{\varrho}_m\Vert_{H^1(\Omega_R)}+\Vert\Psi^{\varrho}_m\Vert_{H^1(\Omega_R)}
\leq C\left(\Vert F^{\varrho}_1\Vert_{L^2(\Omega_R)}+\Vert F^{\varrho}_2\Vert_{L^2(\Omega_R)}
+\Vert g\Vert_{L^2([-\theta_0,\theta_0])}\right),
\end{equation}
where the constant $ C>0 $  is independent of $\varrho$ and $m$.
\par In view of the orthonormality of the set $\mathcal{A}$ in $L^2(\Gamma)$, the system \eqref{Leqs4}-\eqref{Lbcs4} can be transformed into
\begin{equation}\label{Leqs5}
\begin{split}
\langle F^{\varrho}_1,\eta_k\rangle=&\sum_{j=0}^{m}\langle a^{\varrho}_{11}\eta_j,\eta_k\rangle(\vartheta^{\varrho}_j)''(r)
+\langle 2a^{\varrho}_{12}\eta_j',\eta_k\rangle(\vartheta^{\varrho}_j)'(r)+\bar{a}_1(\vartheta^{\varrho}_k)'(r)\\
&-\left(\frac{j\pi}{\theta_0}\right)^2\langle a^{\varrho}_{22}\eta_j,\eta_k\rangle\vartheta^{\varrho}_j(r)+\langle a^{\varrho}_{2}\eta_j',\eta_k\rangle\vartheta^{\varrho}_j(r)
+\bar{b}_2(\nu^{\varrho}_k)'(r)+\bar{b}_1\nu^{\varrho}_k(r),\\
\langle F^{\varrho}_2,\eta_k\rangle=&\bigg((\hat r\nu^{\varrho}_k)'\bigg)'(r)
-\frac{1}{\hat{r}}\left(\frac{k\pi}{\theta_0}\right)^2\nu^{\varrho}_k(r)+\bar c_2\nu^{\varrho}_k(r)
+\bar c_1(\vartheta^{\varrho}_k)'(r),
\end{split}
\end{equation}
for each $r\in(0,R)$, $m\in\mathbb{N}$ and all $k=0,1,...,m$,
with the boundary conditions
\begin{equation}\label{Lbcs5}
\begin{cases}
\begin{aligned}
&(\vartheta^{\varrho}_k(0),(\nu^{\varrho}_k)'(0))=(0,0),\\
&(\vartheta^{\varrho}_k)'(0)=\langle g,\eta_k\rangle,\\
&\nu^{\varrho}_k(R)=0.
\end{aligned}
\end{cases}
\end{equation}
For each $m\in\mathbb{N}$, set
\begin{equation*}
\begin{split}
\mathbf{X}_1&:=(\vartheta^{\varrho}_0,...,\vartheta^{\varrho}_m),\quad
\mathbf{X}_2:=\mathbf{X}^{\prime}_1=((\vartheta^{\varrho}_0)',...,(\vartheta^{\varrho}_m)'),\\
\mathbf{X}_3&:=(\nu^{\varrho}_0,...,\nu^{\varrho}_m),\quad
\mathbf{X}_4:=\mathbf{X}^{\prime}_3=((\nu^{\varrho}_0)',...,(\nu^{\varrho}_m)'),\quad
\mathbf{X}:=(\mathbf{X}_1,\mathbf{X}_2,\mathbf{X}_3,\mathbf{X}_4)^{T}.
\end{split}
\end{equation*}
Furthermore, we define a projection mapping $\Pi_{i}(i=1,2)$ by
\begin{equation*}
\Pi_1(\mathbf{X})=(\mathbf{X}_1,\mathbf{X}_2,\mathbf{0},\mathbf{X}_4)^{T} \quad \hbox{and} \quad
\Pi_2(\mathbf{X})=(\mathbf{0},\mathbf{0},\mathbf{X}_3,\mathbf{0})^{T}.
\end{equation*}
Then \eqref{Leqs5} and \eqref{Lbcs5} can be reduced to a first order ODE system
\begin{equation*}
\mathbf{X}^{\prime}=\mathbb{A}\mathbf{X}+\mathbf{F}
\end{equation*}
with
\begin{equation*}
\Pi_1(\mathbf{X})(0)=(\mathbf{0},\mathbf{0},\sum_{j=0}^{m}\langle g,\eta_j\rangle\mathbf{e}_j,\mathbf{0}):=\mathbf{P}_0,
\quad \Pi_2(\mathbf{X})(R)=\mathbf{0}
\end{equation*}
for the smooth mappings $\mathbb{A}:(0,R)\rightarrow \mathbb{R}^{4(m+1)\times4(m+1)}$ and $\mathbf{F}:(0,R)\rightarrow \mathbb{R}^{4(m+1)}$ with respect to the coefficients and right terms of \eqref{Leqs5}-\eqref{Lbcs5}.
We can obtain $\mathbf{X}$ by solving the integral equations
\begin{equation}\label{X}
\begin{split}
&\Pi_1(\mathbf{X})(0)-\int_{0}^r\Pi_1\mathbb{A}(\mathbf{X})(t)\mathrm{d}t
=\mathbf{P}_0+\int_{0}^r\Pi_1\mathbf{F}(t)\mathrm{d}t,\\
&\Pi_2(\mathbf{X})(r)-\int_{R}^r\Pi_2\mathbb{A}(\mathbf{X})(t)\mathrm{d}t
=\int_{R}^r\Pi_2\mathbf{F}(t)\mathrm{d}t.
\end{split}
\end{equation}
We define a linear operator $\mathfrak{R}:C^1([0,R];\mathbb{R}^{4(m+1)})\rightarrow C^1([0,R];\mathbb{R}^{4(m+1)})$ by
\begin{equation*}
\mathfrak{R}\mathbf{X}(r):=\Pi_1\int_{0}^r\mathbb{A}(\mathbf{X})(t)\mathrm{d}t
+\Pi_2\int_{R}^r\mathbb{A}(\mathbf{X})(t)\mathrm{d}t.
\end{equation*}
It follows from the Arzel$\grave{\mathrm{a}}$-Ascoli theorem that $\mathfrak{R}$ is compact. Then $\mathbf{X}$ solves \eqref{X} if and only if
\begin{equation}\label{R}
(\mathbf{Id}-\mathfrak{R})\mathbf{X}(r)=\mathbf{P}_0+\Pi_1\int_{0}^r\mathbf{F}(t)\mathrm{d}t
+\Pi_2\int_{R}^r\mathbf{F}(t)\mathrm{d}t.
\end{equation}
It follows from \eqref{H1} and the Fredholm alternative theorem that \eqref{R} has a unique solution. \\
\textbf{Step 3: The well-posedness of the approximated problem}
The bootstrap argument for \eqref{X} gives the smoothness of the solution $\mathbf{X}$ on $[0,R]$. Furthermore, Via the standard theory for elliptic and hyperbolic equations, the estimates of $\Vert\varphi^{\varrho}_m\Vert_{H^4(\Omega_R)}$ and $\Vert\Psi^{\varrho}_m\Vert_{H^4(\Omega_R)}$ can be obtained in the following lemma.
\begin{lemma}\label{3.4}
For each $m\in\mathbb{N}$, if $\varphi^{\varrho}_m$ and $\Psi^{\varrho}_m$ given by \eqref{m} solve the problem \eqref{Leqs4}-\eqref{Lbcs4}, then
\begin{equation}\label{mH^4_1}
\Vert\varphi^{\varrho}_m\Vert_{H^4(\Omega_R)}
\leq C\left(\Vert F^{\varrho}_1\Vert_{H^3(\Omega_R)}+\Vert F^{\varrho}_2\Vert_{H^2(\Omega_R)}
+\Vert g\Vert_{H^3([-\theta_0,\theta_0])}\right)
\end{equation}
and
\begin{equation}\label{mH^4_2}
\Vert\Psi^{\varrho}_m\Vert_{H^4(\Omega_R)}
\leq C\left(\Vert F^{\varrho}_1\Vert_{H^2(\Omega_R)}+\Vert F^{\varrho}_2\Vert_{H^2(\Omega_R)}
+\Vert g\Vert_{H^2([-\theta_0,\theta_0])}\right),
\end{equation}
where $C$ is a constant depending only on background data and $\epsilon_0$.
\end{lemma}
Lemma \ref{3.4} can be established by adjusting arguments of Lemma 2.4 in \cite{Duan0}, and its proof is  omitted here.
Next, for any fixed $\varrho>0$, the sequence $\{(\varphi^{\varrho}_m,\Psi^{\varrho}_m)\}_{m=1}^{\infty}$ is bounded in $[H^4(\Omega_R)]^2$. Since $\eta_j'(\pm\theta_0)=\eta_j'''(\pm\theta_0)=0$ for $j=0,1,\cdot\cdot\cdot$, each $(\varphi^{\varrho}_m,\Psi^{\varrho}_m)$ satisfies the compatibility conditions
\begin{equation*}
\frac{\partial^k\varphi^{\varrho}_m}{\partial\theta^k}=\frac{\partial^k\Psi^{\varrho}_m}{\partial\theta^k}=0
\quad\hbox{on}\quad\Gamma_{w}\quad\hbox{for}\quad k=1,3.
\end{equation*}
By applying the Morrey inequality, the sequence $\{(\varphi^{\varrho}_m,\Psi^{\varrho}_m)\}_{m\in\mathbb{N}}$ is bounded in $[C^{2,\frac{2}{3}}(\overline{\Omega_R})]^2$. Employing the Arzel$\mathrm{\grave{a}}$-Ascoli theorem and the weak compactness property of $H^4(\Omega_R)$, there exists a subsequence $\{(\varphi^{\varrho}_{m_k},\Psi^{\varrho}_{m_k})\}$ and the function $(\varphi^{\varrho}_*,\hat\Psi^{\varrho}_*)$ such that
\begin{equation*}
\begin{split}
&(\varphi^{\varrho}_{m_k},\Psi^{\varrho}_{m_k})\rightarrow(\varphi^{\varrho}_*,\hat\Psi^{\varrho}_*)\quad \hbox{in} \quad [C^{2,\frac{1}{2}}(\overline{\Omega_R})]^2,\\
&(\varphi^{\varrho}_{m_k},\Psi^{\varrho}_{m_k})\rightharpoonup(\varphi^{\varrho}_*,\hat\Psi^{\varrho}_*) \quad \hbox{in} \quad [H^4(\Omega_R)]^2.
\end{split}
\end{equation*}
Furthermore,
\begin{equation}\label{*c}
\frac{\partial^k\varphi^{\varrho}_*}{\partial\theta^k}=\frac{\partial^k\Psi^{\varrho}_*}{\partial\theta^k}=0 \quad \hbox{on} \quad \Gamma_{w} \quad \hbox{for} \quad k=1,3.
\end{equation}
This implies that $(\varphi^{\varrho}_*,\Psi^{\varrho}_*)$ is a classical solution to \eqref{Leqs3}-\eqref{Lbcs3}. Moreover, Lemma \ref{3.4} and the $H^4$ convergence of $\{(\varphi^{\varrho}_{m_k},\Psi^{\varrho}_{m_k})\}$ deduce that
\begin{equation}\label{*H^4_1}
\Vert\varphi^{\varrho}_*\Vert_{H^4(\Omega_R)}
\leq C\left(\Vert F^{\varrho}_1\Vert_{H^3(\Omega_R)}+\Vert F^{\varrho}_2\Vert_{H^2(\Omega_R)}
+\Vert g\Vert_{H^3([-\theta_0,\theta_0])}\right)
\end{equation}
and
\begin{equation}\label{*H^4_2}
\Vert\Psi^{\varrho}_*\Vert_{H^4(\Omega_R)}
\leq C\left(\Vert F^{\varrho}_1\Vert_{H^2(\Omega_R)}+\Vert F^{\varrho}_2\Vert_{H^2(\Omega_R)}
+\Vert g\Vert_{H^2([-\theta_0,\theta_0])}\right),
\end{equation}
where $C$ is a positive constant depending only on background data and $\epsilon_0$.\\
\textbf{Step 4: The well-posedness of of the linear boundary value problem.}
It follows from \eqref{*c}, \eqref{*H^4_1}, \eqref{*H^4_2} and the Morrey inequality that $\{(\varphi^{\varrho}_*,\Psi^{\varrho}_*)\}$ is bounded in $[C^{2,\frac{2}{3}}(\overline{\Omega_R})]^2$ for $\varrho>0$. By adjusting the limiting argument in Step 2, we extract a subsequence $\{(\varphi^{\varrho_{k}}_*,\Psi^{\varrho_{k}}_*)\}$ such that its limit $(\varphi_*,\Psi_*):=(\varphi,\Psi)\in[H^4(\Omega_R)\cap C^{2,\frac{1}{2}}(\overline{\Omega_R})]^2$ is a classical solution to \eqref{elli4}-\eqref{aBC4}, and satisfies the estimates \eqref{H^4_1} and \eqref{H^4_2}.
This completes the proof of Proposition \ref{prop2}.
\end{proof}
The well-posedness of \eqref{elli1P} and \eqref{aBC1P} directly follows from Proposition \ref{prop2}.
\begin{corollary}\label{lem3}
Fix $\epsilon_0\in(0,\mathcal{R})$. Let $\bar R\in(0,\mathcal{R}-\epsilon_0]$ be given in Proposition \ref{prop1}. Then for $R\in(0,\bar R)$ and $\max\{\delta_{\vartheta},\delta_{\nu}\}\leq \delta_0$, there exists a constant $C$ depending only on the background data and $\epsilon_0$ such that such that the boundary value problem \eqref{elli1P}-\eqref{aBC1P} associated with $(\mathcal{S}^*,\mathcal{K}^*)\in \mathcal{I}_{\delta_{\vartheta},R}$ and $(\mathcal{U}^*,\mathcal{V}^*,\check{\Phi}^*)\in \mathcal{I}_{\delta_{\nu},R}$ has a unique solution $(\mathcal{U},\mathcal{V},\check{\Phi})\in [H^3(\Omega_R)]^2\times H^4(\Omega_R)$ satisfying
\begin{equation}\label{zH^41}
\Vert(\mathcal{U},\mathcal{V})\Vert_{H^3(\Omega_R)}+\Vert\check{\Phi}\Vert_{H^4(\Omega_R)}
\leq C\left(\Vert (f_1,f_2)\Vert_{H^3(\Omega_R)}+\Vert f_3\Vert_{H^2(\Omega_R)}
+\sigma_2(U_{en}, V_{en}, E_{en}, \Phi_{ex})\right).
\end{equation}
Furthermore, the solution $(\mathcal{U},\mathcal{V},\check{\Phi})$ satisfies the compatibility conditions
\begin{equation}\label{compatibility2}
\frac{\partial\mathcal{U}}{\partial\theta}=\frac{\partial^{k-1}\mathcal{V}}{\partial\theta^{k-1}}=\frac{\partial^k\check{\Phi}}{\partial\theta^k}=0 \quad \hbox{on} \quad \Gamma_{w} \quad \hbox{for} \quad k=1,3.
\end{equation}
\end{corollary}

\section{Nonlinear stability}\noindent
\par This section is devoted to the nonlinear structural stability of supersonic flows with nonzero vorticity to the Euler-Poisson system. We first outline the iteration scheme, then give the detailed proofs of Theorem \ref{thm1} in Section 4.2.
\par Suppose that $\max\{\delta_{\mu},\delta_{\nu}\}\leq \delta_0$ and $R\in(0,\bar R)$ hold. Fix $Q=(\mathcal{S}^*,\mathcal{K}^*)\in \mathcal{I}_{\delta_{\mu},R}$. For each $\bm{\mathcal{P}}=(\mathcal{U}^*,\mathcal{V}^*,\check{\Phi}^*)$, let $(\mathcal{U},\mathcal{V},\check{\Phi})$ be the unique solution to the linear boundary value problem \eqref{elli1P}-\eqref{aBC1P} associated with $(\mathcal{S}^*,\mathcal{K}^*,\mathcal{U}^*,\mathcal{V}^*,\check{\Phi}^*)\in \mathcal{I}_{\delta_{\mu},R}\times \mathcal{I}_{\delta_{\nu},R}$. Therefore, one can define a mapping $\mathcal{T}^Q_1: \mathcal{I}_{\delta_{\nu},R}\mapsto [H^3(\Omega_R)]^2\times H^4(\Omega_R)$ by
\begin{equation}\label{iteration1}
\mathcal{T}^Q_1(\mathcal{U}^*,\mathcal{V}^*,\check{\Phi}^*)=(\mathcal{U},\mathcal{V},\check{\Phi}).
\end{equation}
We express $\delta_\nu$ in terms of $\delta_\mu$ and $\sigma$. If there exist positive constant $\hat{\mu}$ such that, whenever $\delta_{\mu}+\sigma<\hat{\mu}$, the iteration mapping $\mathcal{T}^Q_1$ has a unique fixed point in $\mathcal{I}_{\delta_{\nu},R}$ for each $Q\in \mathcal{I}_{\delta_{\mu},R}$.
Next, for each $Q\in \mathcal{I}_{\delta_{\mu},R}$, let $(\mathcal{U}^Q,\mathcal{V}^Q,\check{\Phi}^Q)\in \mathcal{I}_{\delta_{\nu},R}$ be the fixed point of the iteration mapping $\mathcal{T}^Q_1$. We will show that the transport equations in terms of $(\mathcal{S},\mathcal{K})$ with the corresponding boundary conditions in \eqref{elli1} and \eqref{aBC1} has a unique solution $(\mathcal{S},\mathcal{K})\in[H^4(\Omega_R)]^2$. Hence, one can define another $\mathcal{T}_2: \mathcal{I}_{\delta_{\mu},R}\mapsto [H^4(\Omega_R)]^2$ by
\begin{equation}\label{iteration2}
\mathcal{T}_2(\mathcal{S}^*,\mathcal{K}^*)=(\mathcal{S},\mathcal{K}).
\end{equation}
Then we express $\delta_\mu$ in terms of $\sigma$, and prove that $\mathcal{T}_2$ has a unique fixed point $Q^{\sharp}=(\mathcal{S}^\sharp,\mathcal{K}^\sharp)\in\mathcal{I}_{\delta_{\mu},R}$. Consequently, $(\mathcal{S}^\sharp,\mathcal{K}^\sharp,\mathcal{U}^{Q^{\sharp}},\mathcal{V}^{Q^{\sharp}},\check{\Phi}^{Q^{\sharp}})$ solves the nonlinear boundary value problem \eqref{elli1} and \eqref{aBC1}. Finally, we prove that if $\sigma$ is sufficiently small, there exists a unique solution to Problem \ref{pro1}.

\par In the following, we give the detailed proof for Theorem \ref{thm1} outlined above.\\
\textbf{Step 1: The existence and uniqueness of a fixed point of $\mathcal{T}^Q_1$.}
\par We fixe $\epsilon_0\in(0,\frac{1}{10}\mathcal{R})$ and $R\in(0,\bar R)$. Assume that
\begin{equation}\label{NNNmunu}
\max\{\delta_{\mu},\delta_{\nu}\}\leq \delta_0,
\end{equation}
where $\delta_0$ and $\bar R$ are given in Lemma \ref{pr2} and Proposition \ref{prop1}, respectively.
Set
\begin{equation}\label{NNNnu}
\delta_{\nu}=m_1(\delta_{\mu}+\sigma)
\end{equation}
for the constant $m_1$ to be specified later.
Fix $Q\in \mathcal{I}_{\delta_{\mu},R}$, and let $\mathcal{T}^Q_1: \mathcal{I}_{\delta_{\nu},R}\mapsto [H^3(\Omega_R)]^2\times H^4(\Omega_R)$ given by \eqref{iteration1}. Fix $(\mathcal{U}^*,\mathcal{V}^*,\check{\Phi}^*)\in\mathcal{I}_{\delta_{\nu},R}$, and set $(\mathcal{U},\mathcal{V},\check{\Phi})=\mathcal{T}^Q_1(\mathcal{U}^*,\mathcal{V}^*,\check{\Phi}^*)$. It follows from \eqref{4iter1}, \eqref{regu},  and \eqref{zH^41} that there exists a constant $C_1^*$ depending only on background data and $\epsilon_0$ such that
\begin{equation*}
\Vert(\mathcal{U},\mathcal{V})\Vert_{H^3(\Omega_R)}+\Vert\check{\Phi}\Vert_{H^4(\Omega_R)}
\leq C^*_1(\delta_{\mu}+\delta_{\nu}^2+\sigma)<\frac{1}{2}\delta_{\nu},
\end{equation*}
where we choose $m=4C^*_1$ and suppose that
\begin{equation}\label{Ndelta1}
\delta_{\mu}+\sigma\leq \frac{1}{16(C_1^*)^2}:=\hat{\mu}.
\end{equation}
Therefore, $\mathcal{T}^Q_1$ maps  $\mathcal{I}_{\delta_{\nu},R}$ into itself for any $Q\in \mathcal{I}_{\delta_{\mu},R}$.
\par Next, for any fixed $\mathcal{U}^*,\mathcal{V}^*,\check{\Phi}^*\in\mathcal{I}_{\delta_{\nu},R}$, we will show that $\mathcal{T}^Q_1$  is a contraction mapping in a low order norm $\Vert(\mathcal{U},\mathcal{V})\Vert_{L^2(\Omega_R)}+\Vert\check{\Phi}\Vert_{H^1(\Omega_R)}$.
Let $\bm{\mathcal{P}}^i=(\mathcal{U}^*_i,\mathcal{V}^*_i,\check{\Phi}^*_i)\in\mathcal{I}_{\delta_{\nu},R}$, $i=1,2$, then $\mathcal{T}^Q_1(\mathcal{U}^*_i,\mathcal{V}^*_i,\check{\Phi}^*_i)=(\mathcal{U}_i,\mathcal{V}_i,\check{\Phi}_i)$. Set
\begin{equation*}
(Y^*_1,Y^*_2,Y^*_3)=(\mathcal{U}^*_1,\mathcal{V}^*_1,\check{\Phi}^*_1)-(\mathcal{U}^*_2,\mathcal{V}^*_2,\check{\Phi}^*_2),\quad
(Y_1,Y_2,Y_3)=(\mathcal{U}_1,\mathcal{V}_1,\check{\Phi}_1)-(\mathcal{U}_2,\mathcal{V}_2,\check{\Phi}_2).
\end{equation*}
It follows from \eqref{elli1P} and \eqref{aBC1P} that $(Y_1,Y_2,Y_3)$ satisfies the system
\begin{equation}\label{Nelli1P}
\begin{cases}
\begin{aligned}
&L_1^{P^1}(Y_1,Y_2,Y_3)=\mathcal{F}_1(r,\theta;\mathcal{K}^*,\mathcal{U}^*_i,\mathcal{V}^*_i,\check{\Phi}^*_i),\\
&L_2(Y_1,Y_3)=\mathcal{F}_2(r,\theta;\mathcal{S}^*,\mathcal{K}^*,\mathcal{U}^*_i,\mathcal{V}^*_i,\check{\Phi}^*_i),\\
&(\hat{r}Y_2)_r-\partial_{\theta}Y_1=\mathcal{F}_3(r,\theta;\mathcal{S}^*,\mathcal{K}^*,\mathcal{U}^*_i,\mathcal{V}^*_i,\check{\Phi}^*_i),\\
\end{aligned}
\end{cases}
\end{equation}
and the boundary conditions
\begin{equation}\label{NaBC1P}
\begin{cases}
(Y_1,Y_2,\partial_rY_3)(0,\theta)=(0,0,0),\ \ &{\rm{on}}\ \ \Gamma_{en},\\
Y_3(R,\theta)=0,\ \ &{\rm{on}}\ \ \Gamma_{ex},\\
Y_2(r,\pm\theta_0)=\partial_{\theta}Y_3(r,\pm\theta_0)=0, \ \ &{\rm{on}}\ \ \Gamma_w,
\end{cases}
\end{equation}
where
\begin{equation*}
\begin{split}
&L_1^{P^1}(Y_1,Y_2,Y_3)=
\hat{a}_{11}^{P^1}\partial_rY_1
+\hat{a}_{22}^{P^1}\partial_{\theta}Y_2
+\hat{a}_{12}^{P^1}\partial_{\theta}Y_1
+\hat{a}_{21}^{P^1}\partial_{r}Y_2
+\bar{a}_1Y_1+\bar{b}_1Y_3+\bar{b}_2\partial_rY_3,\\
&L_2(Y_1,Y_3)=(\hat r\partial_{r}Y_3)_r+\frac{1}{\hat{r}}\partial_{\theta\theta}Y_3
+\bar{c}_1(r)Y_1+\bar{c}_2(r)Y_3
\end{split}
\end{equation*}
and
\begin{equation*}
\begin{split}
\mathcal{F}_1(r,\theta;\mathcal{K}^*,\mathcal{U}^*_i,\mathcal{V}^*_i,\check{\Phi}^*_i)=&
-\bigg((\hat{a}_{11}^{P^1}-\hat{a}_{11}^{P^2})\partial_r\mathcal{U}_2
+(\hat{a}_{22}^{P^1}-\hat{a}_{22}^{P^2})\partial_{\theta}\mathcal{V}_2   \\
&+(\hat{a}_{12}^{P^1}-\hat{a}_{12}^{P^2})\partial_{\theta}\mathcal{U}_2
+(\hat{a}_{21}^{P^1}-\hat{a}_{21}^{P^2})\partial_{r}\mathcal{V}_2\bigg)  \\
&+f_1(r,\theta;\mathcal{K}^*,\mathcal{U}^*_1,\mathcal{V}^*_1,\check{\Phi}^*_1)
-f_1(r,\theta;\mathcal{K}^*,\mathcal{U}^*_2,\mathcal{V}^*_2,\check{\Phi}^*_2),\\
\mathcal{F}_2(r,\theta;\mathcal{S}^*,\mathcal{K}^*,\mathcal{U}^*_i,\mathcal{V}^*_i,\check{\Phi}^*_i)=&
f_2(r,\theta;\mathcal{S}^*,\mathcal{K}^*,\mathcal{U}^*_1,\mathcal{V}^*_1,\check{\Phi}^*_1)
-f_2(r,\theta;\mathcal{S}^*,\mathcal{K}^*,\mathcal{U}^*_2,\mathcal{V}^*_2,\check{\Phi}^*_2),\\
\mathcal{F}_3(r,\theta;\mathcal{S}^*,\mathcal{K}^*,\mathcal{U}^*_i,\mathcal{V}^*_i,\check{\Phi}^*_i)=&
f_3(r,\theta;\mathcal{S}^*,\mathcal{K}^*,\mathcal{U}^*_1,\mathcal{V}^*_1,\check{\Phi}^*_1)
-f_3(r,\theta;\mathcal{S}^*,\mathcal{K}^*,\mathcal{U}^*_2,\mathcal{V}^*_2,\check{\Phi}^*_2).
\end{split}
\end{equation*}
 Then there exist $\tilde{\phi}$ and $\tilde{\psi}$ such that $Y_1$ and $Y_2$ can be decomposed as
\begin{equation*}
Y_1=-\frac{1}{\hat{r}}\tilde{\phi}_{\theta}+\tilde{\psi}_r,\quad Y_2=\tilde{\phi}_{r}+\frac{1}{\hat{r}}\tilde{\psi}_{\theta}.
\end{equation*}
Moreover, $\tilde{\phi}$ satisfies
\begin{equation}\label{Nextend}
\begin{cases}
(\hat r\tilde{\phi}_{r})_r+\frac{1}{\hat{r}}\tilde{\phi}_{\theta\theta}
=\mathcal{F}_3,   &{\rm{in}}\quad \Omega_R,\\
\tilde{\phi}_r(0,\theta)=0,   &{\rm{on}}\quad \Gamma_{en},\\
\tilde{\phi}_r(R,\theta)=0,   &{\rm{on}}\quad \Gamma_{ex},\\
\tilde{\phi}(r,\pm \theta_0)=0,   &{\rm{on}}\quad \Gamma_w,
\end{cases}
\end{equation}
and $\tilde{\psi}$ satisfies
\begin{equation}
\begin{cases}
\mathcal{L}_1^{\A}(\tilde{\psi},Y_3)=\tilde{\mathcal{F}}_1,  &\rm{in}\quad \Omega_R,\\
\mathcal{L}_2(\tilde{\psi},Y_3)=\tilde{\mathcal{F}}_2,  &\rm{in}\quad \Omega_R,\\
(\partial_{r}\tilde{\psi},\partial_{\theta}\tilde{\psi},\partial_rY_3)(0,\theta)
=(\frac{1}{r_2}\partial_{\theta}\tilde{\phi}(0,\theta),0,0),  &\rm{on}\quad \Gamma_{en},\\
Y_3(R,\theta)=0,  &\rm{on}\quad \Gamma_{ex},\\
\partial_{\theta}\tilde{\psi}(r,\pm\theta_0)=\partial_{\theta}Y_3(r,\pm\theta_0)=0,  &\rm{on}\quad \Gamma_w,
\end{cases}
\end{equation}
where
\begin{equation*}
\begin{split}
&\mathcal{L}_1^{\A}(\tilde{\psi},Y_3)=
a_{11}^{P_1}\partial_{rr}\tilde{\psi}
+a_{22}^{P_1}\partial_{\theta\theta}\tilde{\psi}
+2a_{12}^{P_1}\partial_{r\theta}\tilde{\psi}
+\bar{a}_1\partial_{r}\tilde{\psi}+ a_2^{P_1} \partial_{\theta}\tilde{\psi}+
\bar{b}_1Y_3+\bar{b}_2\partial_rY_3,\\
&\mathcal{L}_2(\tilde{\psi},Y_3)=(\hat r\partial_{r}Y_3)_r+\frac{1}{\hat{r}}\partial_{\theta\theta}Y_3
+\bar{c}_1\partial_{r}\tilde{\psi}+\bar{c}_2Y_3\\
&a_{11}^{P_1}=\hat{a}_{11}^{P_1},\quad a_{22}^{P_1}=\frac{1}{\hat{r}}\hat{a}_{22}^{P_1},\quad a_{12}^{P_1}=\hat{a}_{12}^{P_1},\quad
a_2^{P_1}=\frac{1}{\hat{r}}\hat{a}_{12}^{P_1},\\
&\tilde{\mathcal{F}}_1=\mathcal{F}_1-\hat{a}_{21}^{P_1}\partial_{rr}\tilde{\phi}
+\frac{1}{\hat{r}}\hat{a}_{12}^{P_1}\partial_{\theta\theta}\tilde{\phi}
+\left(\frac{1}{\hat{r}}\hat{a}_{11}^{P_1}
-\hat{a}_{22}^{P_1}\right)\partial_{r\theta}\tilde{\phi}
+\left(\frac{1}{\hat{r}^2}\hat{a}_{11}^{P_1}
+\frac{1}{\hat{r}}\bar a_1\right)\partial_{\theta}\tilde{\phi},\\
&\tilde{\mathcal{F}}_2=\mathcal{F}_2+\frac{1}{\hat{r}}\bar c_1\partial_{\theta}\tilde{\phi}.
\end{split}
\end{equation*}
The standard elliptic theory in \cite{Gilbarg} implies that
\begin{equation*}
\Vert\tilde{\phi}\Vert_{H^2(\Omega_R)}\leq C\Vert \mathcal{F}_3\Vert_{L^2(\Omega_R)}
\leq C(\delta_{\mu}+\delta_{\nu})\left(\Vert (Y_1^*,Y_2^*)\Vert_{L^2(\Omega_R)}+\Vert Y_3^*\Vert_{H^1(\Omega_R)}\right)
\end{equation*}
This together with Proposition \ref{prop1} gives that
\begin{equation*}
\begin{split}
\Vert(\tilde{\psi},Y_3)\Vert_{H^1(\Omega_R)}
&\leq C\left(\Vert \tilde{\mathcal{F}}_1\Vert_{L^2(\Omega_R)}+\Vert \tilde{\mathcal{F}}_2\Vert_{L^2(\Omega_R)}
+\Vert \partial_{\theta}\tilde{\phi}(0,\cdot)\Vert_{L^2([-\theta_0,\theta_0])}\right)\\
&\leq C_2^*(\delta_{\mu}+\delta_{\nu})\left(\Vert (Y_1^*,Y_2^*)\Vert_{L^2(\Omega_R)}+\Vert Y_3^*\Vert_{H^1(\Omega_R)}\right),
\end{split}
\end{equation*}
where $C_2^*$ is a constant depending only on the background data and $\epsilon_0$.
Therefore,
\begin{equation}\label{Ndelta2-2}
\begin{split}
\Vert(Y_1,Y_2)\Vert_{L^2(\Omega_R)}+\Vert Y_3\Vert_{H^1(\Omega_R)}
&\leq C_2^*(\delta_{\mu}+\delta_{\nu})\left(\Vert (Y_1^*,Y_2^*)\Vert_{L^2(\Omega_R)}+\Vert Y_3^*\Vert_{H^1(\Omega_R)}\right)\\
&\leq 4C_2^*(C_1^*+1)(\delta_{\mu}+\sigma)\left(\Vert (Y_1^*,Y_2^*)\Vert_{L^2(\Omega_R)}+\Vert Y_3^*\Vert_{H^1(\Omega_R)}\right).
\end{split}
\end{equation}
If $(\delta_{\mu},\sigma)$ satisfy
\begin{equation}\label{Ndelta2}
4C_2^*(C_1^*+1)(\delta_{\mu}+\sigma)\leq\frac{1}{8}
\end{equation}
and \eqref{Ndelta1}, then the contraction mapping theorem implies that there exists a unique fixed point $(\mathcal{U}^Q,\mathcal{V}^Q,$\\$\check{\Phi}^Q)$ for the iteration mapping $\mathcal{T}^Q_1$.\\
\textbf{Step 2: The existence and uniqueness of a fixed point of $\mathcal{T}_2$.}
\par For any fixed $Q=(\mathcal{S}^*,\mathcal{K}^*)\in \mathcal{I}_{\delta_{\mu},R}$, the above arguments imply that $\mathcal{T}^Q_1$ has the unique fixed point $(\mathcal{U}^Q,\mathcal{V}^Q,\check{\Phi}^Q)\in\mathcal{I}_{\delta_{\nu},R}$, provided \eqref{Ndelta1} and \eqref{Ndelta2} hold. Furthermore, $(\mathcal{U}^Q,\mathcal{V}^Q,\check{\Phi}^Q)$ solves the nonlinear equations
\begin{equation}\label{NNelli1}
\begin{cases}
\begin{aligned}
&L_1(\mathcal{U}^Q,\mathcal{V}^Q,\check{\Phi}^Q)=f_1(r,\theta;\mathcal{K}^*,\mathcal{U}^Q,\mathcal{V}^Q,\check{\Phi}^Q),\\
&L_2(\mathcal{U}^Q,\check{\Phi}^Q)=f_2(r,\theta;\mathcal{S}^*,\mathcal{K}^*,\mathcal{U}^Q,\mathcal{V}^Q,\check{\Phi}^Q),\\
&(\hat{r}\mathcal{V}^Q)_r-\mathcal{U}^Q_{\theta}=f_3(r,\theta;
\mathcal{S}^*,\mathcal{K}^*,\mathcal{U}^Q,\mathcal{V}^Q,\check{\Phi}^Q),
\end{aligned}
\end{cases}
\end{equation}
in $\Omega_R$, subject to the boundary conditions
\begin{equation}\label{NNaBC1}
\begin{cases}
(\mathcal{U}^Q,\mathcal{V}^Q,\check{\Phi}^Q_r)(0,\theta)
=(U_{en}-U_0,V_{en},-E_{en}+E_0)(\theta),\ \ &{\rm{on}}\ \ \Gamma_{en},\\
\check{\Phi}^Q(R,\theta)=\Phi_{ex}(\theta)-\bar\Phi(R),\ \ &{\rm{on}}\ \ \Gamma_{ex},\\
\mathcal{V}^Q(r,\pm\theta_0)=\check{\Phi}^Q_{\theta}(r,\pm\theta_0)=0, \ \ &{\rm{on}}\ \ \Gamma_w,
\end{cases}
\end{equation}
Thus,
\begin{equation*}
(U^Q,V^Q,\Phi^Q)(r,\theta)=(\mathcal{U}^Q,\mathcal{V}^Q,\check{\Phi}^Q)
+(\bar U(r),0,\bar\Phi(r))\in [H^3(\Omega_R)]^2\times H^4(\Omega_R)
\end{equation*}
solves the system
\begin{equation}\label{Nr1EP1}
\begin{cases}
\begin{aligned}
&A_{11}(\mathcal{K}^*+0,U^Q,V^Q,\Phi^Q)U^Q_r+A_{22}(\mathcal{K}^*+0,U^Q,V^Q,\Phi^Q)V^Q_{\theta}+A_{12}(U^Q,V^Q)U^Q_{\theta}\\
&+A_{21}(U^Q,V^Q)V^Q_{r}
+B(\mathcal{K}^*+0,U^Q,V^Q,\Phi^Q)=0,\\
&(\hat{r}\Phi^Q_{r})_r+\frac{1}{\hat{r}}\Phi^Q_{\theta\theta}
=\hat{r}(\mathcal{H}(\mathcal{S}^*+S_0,\mathcal{K}^*+0,U^Q,V^Q,\Phi^Q)-b),\\
&U^Q((\hat{r}V^Q)_r-U^Q_{\theta})=\frac{\mathrm{e}^{\mathcal{S}^*+S_0}\mathcal{H}^{\gamma-1}(\mathcal{S}^*+S_0,\mathcal{K}^*+0,\Phi^Q,U^Q,V^Q)}
{\gamma-1}\mathcal{S}^*_{\theta}-\mathcal{K}^*_{\theta},\\
\end{aligned}
\end{cases}
\end{equation}
in $\Omega_R$, with the boundary conditions
\begin{equation}
\begin{cases}
(U^Q,V^Q,\Phi^Q_r)(0,\theta)=(U_{en},V_{en},-E_{en})(\theta),\ \ &{\rm{on}}\ \ \Gamma_{en},\\
\Phi^Q(R,\theta)=\Phi_{ex}(\theta),\ \ &{\rm{on}}\ \ \Gamma_{ex},\\
V^Q(r,\pm\theta_0)=\Phi^Q_{\theta}(r,\pm\theta_0)=0, \ \ &{\rm{on}}\ \ \Gamma_w.
\end{cases}
\end{equation}
It follows from \eqref{2rho} that $\mathcal{H}(\mathcal{S}^*+S_0,\mathcal{K}^*+0,U^Q,V^Q,\Phi^Q)\in H^3(\Omega_R)$. Furthermore, the
Morrey inequality yields
\begin{equation}
\Vert \Phi^Q\Vert_{C^{2,\frac{1}{2}}(\overline{\Omega_R})}\leq \Vert \Phi^Q\Vert_{H^4(\Omega_R)},\quad
\Vert \mathcal{H}\Vert_{C^{2,\frac{1}{2}}(\overline{\Omega_R})}\leq \Vert \mathcal{H}\Vert_{H^3(\Omega_R)}.
\end{equation}
We consider the problem
\begin{equation}
\begin{cases}
\begin{aligned}
&(\hat{r}\Phi^Q_{r})_r+\frac{1}{\hat{r}}\Phi^Q_{\theta\theta}
=\hat{r}(\mathcal{H}(\mathcal{S}^*+S_0,\mathcal{K}^*+0,U^Q,V^Q,\Phi^Q)-b),\ \ &{\rm{in}}\ \ \Omega_R,\\
&\Phi^Q_r(0,\theta)=-E_{en}(\theta),\ \ &{\rm{on}}\ \ \Gamma_{en},\\
&\Phi^Q(R,\theta)=\Phi_{ex}(\theta),\ \ &{\rm{on}}\ \ \Gamma_{ex},\\
&\Phi^Q_{\theta}(r,\pm\theta_0)=0, \ \ &{\rm{on}}\ \ \Gamma_w.
\end{aligned}
\end{cases}
\end{equation}
Combining the \eqref{compatibility conditions}, \eqref{1Set} and \eqref{compatibility2} deduces
\begin{equation*}
\partial_{\theta}(\mathcal{H},b)=0,\quad (E_{en},\Phi_{ex})'=0 \quad\rm{on}\quad \Gamma_w.
\end{equation*}
Then the standard Schauder estimates and the method of reflection give that
\begin{equation}\label{NPho}
\Vert \Phi^Q\Vert_{C^{3,\frac{1}{2}}(\overline{\Omega_R})}\leq C\left(\Vert \Phi^Q\Vert_{H^4(\Omega_R)}+\Vert \mathcal{H}\Vert_{H^3(\Omega_R)}+
\Vert b\Vert_{C^2(\overline{\Omega_R})}+\Vert E_{en}\Vert_{C^4(\Gamma_{en})}+\Vert \Phi_{ex}\Vert_{C^4(\Gamma_{ex})}\right).
\end{equation}
\par For any $(\mathcal{S}^*,\mathcal{K}^*)\in \mathcal{I}_{\delta_{\mu},R}$, let $\mathcal{T}_2: \mathcal{I}_{\delta_{\mu},R}\mapsto [H^4(\Omega_R)]^2$ given by \eqref{iteration2}. In view of \eqref{2-x}, $(S,\mathscr{K})$ satisfies
\begin{equation}\label{NSe}
\begin{cases}
\begin{aligned}
&\mathcal{H}(\mathcal{S}^*+S_0,\mathcal{K}^*,U^Q,V^Q,\Phi^Q)\left(U^Q\mathcal{S}_r
+\frac{1}{\hat{r}}V^Q\mathcal{S}_{\theta}\right)=0,
 \ \ {\rm{in}}\ \ \Omega_R,\\
&\mathcal{S}(0,\theta)=S_{en}-S_0, \ \ {\rm{on}}\ \ \Gamma_{en},
\end{aligned}
\end{cases}
\end{equation}
and
\begin{equation}\label{NKe}
\begin{cases}
\begin{aligned}
&\mathcal{H}(\mathcal{S}^*+S_0,\mathcal{K}^*,U^Q,V^Q,\Phi^Q)\left(U^Q\mathcal{K}_r+\frac{1}{\hat{r}} V^Q\mathcal{K}_{\theta}\right)=0,
 \ \ {\rm{in}}\ \ \Omega_R,\\
&\mathcal{K}(0,\theta)=\mathscr{K}_{en},\ \ {\rm{on}}\ \ \Gamma_{en}.
\end{aligned}
\end{cases}
\end{equation}
The first equation in \eqref{Nr1EP1} implies
\begin{equation*}
(\hat{r}\mathcal{H}(\mathcal{S}^*+S_0,\mathcal{K}^*,U^Q,V^Q,\Phi^Q) U^Q)_{r}
+(\mathcal{H}(\mathcal{S}^*+S_0,\mathcal{K}^*,U^Q,V^Q,\Phi^Q) V^Q)_{\theta}=0.
\end{equation*}
Then we can define a stream function on $[0,R]\times[-\theta_0,\theta_0]$ as
\begin{equation}
\begin{split}
w(r,\theta)=&\int_{-\theta_0}^{\theta}(r_2\mathcal{H}(\mathcal{S}^*+S_0,\mathcal{K}^*,U^Q,V^Q,\Phi^Q) U^Q)(0,s)\mathrm{d}s\\
&-\int_{0}^{r}(\mathcal{H}(\mathcal{S}^*+S_0,\mathcal{K}^*,U^Q,V^Q,\Phi^Q) V^Q)(s,-\theta_0)\mathrm{d}s.
\end{split}
\end{equation}
Hence, $w(r,\theta)$ satisfies
\begin{equation*}
\begin{cases}
w_r=-(\mathcal{H}(\mathcal{S}^*+S_0,\mathcal{K}^*,U^Q,V^Q,\Phi^Q) V^Q),\\
w_{\theta}=(\hat{r}\mathcal{H}(\mathcal{S}^*+S_0,\mathcal{K}^*,U^Q,V^Q,\Phi^Q) U^Q),
\end{cases}
\end{equation*}
and $(w_r,w_{\theta},w)\in [H^3(\Omega_R)]^2\times H^4(\Omega_R)$.
Since $w_r(r,\pm \theta_0)=0$, then $w(r,\pm \theta_0)=w(0,\pm \theta_0)$. Due to $\bar U(r)>0$ on $(0,R)$, and
$(\mathcal{U}^Q,\mathcal{V}^Q,\check{\Phi}^Q)\in\mathcal{I}_{\delta_{\nu},R}$, one obtains
\begin{equation*}
w_{\theta}(r,\theta)=(\hat{r}\mathcal{H}(\mathcal{S}^*+S_0,\mathcal{K}^*,U^Q,V^Q,\Phi^Q) (\bar U+\mathcal{U}^Q)(r,\theta)>0.
\end{equation*}
Hence, $w(r,\theta)$ is strictly increasing with respect to $\theta$ for each fixed $r\in(0,R)$. This implies that the interval $[w(r,-\theta_0),w(r,\theta_0)]$ is simply equal to $[w(0,-\theta_0),w(0,\theta_0)]$. Denote the inverse function of $w(0,\cdot):[-\theta_0,\theta_0]\rightarrow [w(0,-\theta_0),w(0,\theta_0)]$ by $w^{-1}_0(\cdot):[w(0,-\theta_0),w(0,\theta_0)]\rightarrow[-\theta_0,\theta_0]$. Denote
\begin{equation}\label{deW}
\mathcal{S}(r,\theta)=S_{en}\left(w^{-1}_0(w(r,\theta))\right)-S_0,\quad
\mathcal{K}(r,\theta)=\mathscr{K}_{en}\left(w^{-1}_0(w(r,\theta))\right).
\end{equation}
Note that $(\mathcal{S},\mathcal{K})$ solves \eqref{NSe} and \eqref{NKe}. Moreover, From \eqref{2rho} and \eqref{NPho}, it holds that
\begin{equation*}
\begin{split}
&\left(\mathcal{H}(\mathcal{S}^*+S_0,\mathcal{K}^*,U^Q,V^Q,\Phi^Q) U^Q\right)(0,\cdot)\\
&=\bigg(\bigg(\frac{\gamma-1}{\gamma\mathrm{e}^{S_{en}}}\bigg
(\mathscr{K}_{en}+\Phi^Q-\frac{U_{en}^2+V_{en}^2}{2}\bigg)\bigg)
^{\frac{1}{\gamma-1}} U^Q\bigg)(0,\cdot)
\in C^3([-\theta_0,\theta_0]).
\end{split}
\end{equation*}
Then $w^{-1}_0(\cdot)\in C^4([w(0,-\theta_0),w(0,\theta_0)])$ and
\begin{equation}\label{NSKes}
\Vert\mathcal{S}\Vert_{H^4(\Omega_R)}\leq C_3^*\Vert S_{en}-S_0\Vert_{C^4([-\theta_0,\theta_0])},\quad
\Vert\mathcal{K}\Vert_{H^4(\Omega_R)}\leq C_3^*\Vert \mathscr{K}_{en}\Vert_{C^4([-\theta_0,\theta_0])},
\end{equation}
provided the constant $C_3^*$ depends only on the background data and $\epsilon_0$. By the compatibility conditions \eqref{compatibility conditions}, we have $\partial_{\theta}\mathcal{S}(\pm\theta_0)=\partial_{\theta}\mathcal{K}(\pm\theta_0)=0$ for $k=1,3$.
For each $(\mathcal{S}^*,\mathcal{K}^*)\in \mathcal{I}_{\delta_{\mu},R}$ and the iteration mapping $\mathcal{T}_2$,
set
\begin{equation}\label{NNNmu}
\delta_{\mu}=4C_3^*\sigma.
\end{equation}
Then \eqref{NSKes} gives
\begin{equation}
\Vert(\mathcal{S},\mathcal{K})\Vert_{H^4(\Omega_R)}\leq 2C_3^*\sigma\leq\frac{\delta_{\mu}}{2}.
\end{equation}
Therefore, $\mathcal{T}_2$ maps $\mathcal{I}_{\delta_{\mu},R}$ into itself.
\par Next, we will show that $\mathcal{T}_2$ is a contraction mapping in a low order norm $\Vert(\mathcal{S},\mathcal{K})\Vert_{H^1(\Omega_R)}$ for a suitably small $\sigma$. Let $Q^i=(\mathcal{S}^*_i,\mathcal{K}^*_i)\in\mathcal{I}_{\delta_{\mu},R}$, $i=1,2$, then $\mathcal{T}_2(\mathcal{S}^*_i,\mathcal{K}^*_i)=(\mathcal{S}_i,\mathcal{K}_i)$. Set
\begin{equation*}
(W^*_1,W^*_2)=(\mathcal{S}^*_1,\mathcal{K}^*_1)-(\mathcal{S}^*_2,\mathcal{K}^*_2),\quad
(W_1,W_2)=(\mathcal{S}_1,\mathcal{K}_1)-(\mathcal{S}_2,\mathcal{K}_2)
\end{equation*}
and
\begin{equation*}
(U^Q_i,V^Q_i,\Phi^Q_i)=(\mathcal{U}^Q_i,\mathcal{V}^Q_i,\check{\Phi}^Q_i)
+(\bar U(r),0,\bar\Phi(r)),
\end{equation*}
where $(\mathcal{U}_i^Q,\mathcal{V}_i^Q,\check{\Phi}_i^Q)$ is the unique fixed point of $\mathcal{T}^{Q^i}_1$. It follows from \eqref{deW} that
\begin{equation}
\mathcal{S}_i(r,\theta)=S_{en}\left(W^{(i)}(r,\theta)\right)-S_0,\quad
\mathcal{K}_i(r,\theta)=\mathscr{K}_{en}\left(W^{(i)}(r,\theta)\right),
\end{equation}
where $W^{(i)}(r,\theta)=(w^{(i)}_0)^{-1}(w^{(i)}(r,\theta))$ and
\begin{equation*}
\begin{split}
w^{(i)}(r,\theta)=&\int_{-\theta_0}^{\theta}(r_2\mathcal{H}(\mathcal{S}_i^*+S_0,\mathcal{K}_i^*,U_i^Q,V_i^Q,\Phi_i^Q) U_i^Q)(0,s)\mathrm{d}s\\
&-\int_{0}^{r}(\mathcal{H}(\mathcal{S}_i^*+S_0,\mathcal{K}_i^*,U_i^Q,V_i^Q,\Phi_i^Q) V_i^Q)(s,-\theta_0)\mathrm{d}s.
\end{split}
\end{equation*}
Moreover, $(w^{(i)}_0)^{-1}$ is the inverse function of $w^{(i)}(0,\cdot)$. Thus
\begin{equation}
\vert W_1\vert=\vert \mathcal{S}_1-\mathcal{S}_2\vert
\leq\Vert S_{en}'\Vert_{L^{\infty}([-\theta_0,\theta_0])}\vert W^{(1)}(r,\theta)-W^{(2)}(r,\theta)\vert.
\end{equation}
Since $w^{(i)}_0\left(W^{(i)}(r,\theta)\right)=w^{(i)}(r,\theta)$, then
\begin{equation*}
\begin{split}
&\int_{W^{(2)}(r,\theta)}^{W^{(1)}(r,\theta)}(r_2\mathcal{H}(\mathcal{S}_1^*+S_0,\mathcal{K}_1^*,U_1^Q,V_1^Q,\Phi_1^Q) U_1^Q)(0,s)\mathrm{d}s\\
&=w^{(1)}(r,\theta)-w^{(2)}(r,\theta)
-\int_{-\theta_0}^{W^{(2)}(r,\theta)}\bigg(r_2\mathcal{H}(\mathcal{S}_1^*+S_0,\mathcal{K}_1^*,U_1^Q,V_1^Q,\Phi_1^Q) U_1^Q\\
&-r_2\mathcal{H}(\mathcal{S}_2^*+S_0,\mathcal{K}_2^*,U_2^Q,V_2^Q,\Phi_2^Q) U_2^Q\bigg)(0,s)\mathrm{d}s.
\end{split}
\end{equation*}
Consequently,
\begin{equation*}
\begin{split}
&\mathcal{L}^{(1)}\vert W^{(1)}(r,\theta)-W^{(2)}(r,\theta)\vert\\
&\leq w^{(1)}(r,\theta)-w^{(2)}(r,\theta)
+\int_{-\theta_0}^{\theta_0}\bigg(r_2\mathcal{H}(\mathcal{S}_1^*+S_0,\mathcal{K}_1^*,U_1^Q,V_1^Q,\Phi_1^Q) U_1^Q\\
&-r_2\mathcal{H}(\mathcal{S}_2^*+S_0,\mathcal{K}_2^*,U_2^Q,V_2^Q,\Phi_2^Q) U_2^Q\bigg)(0,s)\mathrm{d}s
\end{split}
\end{equation*}
with
\begin{equation*}
\mathcal{L}^{(i)}=\min_{\theta\in[-\theta_0,\theta_0]}r_2\mathcal{H}(\mathcal{S}_i^*+S_0,\mathcal{K}_i^*,U_i^Q,V_i^Q,\Phi_i^Q)U_i^Q(0,\theta)>0.
\end{equation*}
Note that
\begin{equation*}
W_1^*(0,\theta)=W_2^*(0,\theta)=\left(U^Q_1-U^Q_2\right)(0,\theta)=\left(V^Q_1-V^Q_2\right)(0,\theta)=0.
\end{equation*}
This implies
\begin{equation}\label{W*1}
\begin{split}
\Vert W_1\Vert_{L^2(\Omega_R)}\leq C\sigma\Big(\Vert(W^*_1,W^*_2)\Vert_{L^2(\Omega_R)}+\Vert (Z_1,Z_2,Z_3)\Vert_{L^2(\Omega_R)}
+\Vert Z_3(0,\cdot)\Vert_{L^2([-\theta_0,\theta_0])}\Big),
\end{split}
\end{equation}
where
\begin{equation*}
\begin{split}
(Z_1,Z_2,Z_3)=(\mathcal{U}^Q_1-\mathcal{U}^Q_2,\mathcal{V}^Q_1-\mathcal{V}^Q_2,\check{\Phi}^Q_1-\check{\Phi}^Q_2).
\end{split}
\end{equation*}
Furthermore,
\begin{equation*}
\begin{split}
\vert\partial_rW_1\vert=&\Big\vert S_{en}'\left(W^{(1)}(r,\theta)\right)\partial_rW^{(1)}(r,\theta)-
S_{en}'\left(W^{(2)}(r,\theta)\right)\partial_rW^{(2)}(r,\theta)\Big\vert\\
\leq &\Vert S_{en}''\Vert_{L^{\infty}([-\theta_0,\theta_0])}\Big\vert W^{(1)}(r,\theta)-W^{(2)}(r,\theta)\Big\vert
\frac{\Vert\nabla w^{(1)}(r,\theta)\Vert_{L^{\infty}([-\theta_0,\theta_0])}}{\mathcal{L}^{(1)}}\\
&+\Vert S_{en}'\Vert_{L^{\infty}([-\theta_0,\theta_0])}
\frac{\Vert\nabla w^{(1)}(r,\theta)\Vert_{L^{\infty}([-\theta_0,\theta_0])}}{\mathcal{L}^{(1)}\mathcal{L}^{(2)}}\\
&\times\Big\vert \mathcal{H}(\mathcal{S}_1^*+S_0,\mathcal{K}_1^*,U_1^Q,V_1^Q,\Phi_1^Q) U_1^Q(0,W^{(1)})\\
&\quad-\mathcal{H}(\mathcal{S}_2^*+S_0,\mathcal{K}_2^*,U_2^Q,V_2^Q,\Phi_2^Q) U_1^Q(0,W^{(2)})\Big\vert\\
&+\Vert S_{en}'\Vert_{L^{\infty}([-\theta_0,\theta_0])}
\frac{\vert\nabla w^{(1)}(r,\theta)
-\nabla w^{(2)}(r,\theta)\vert}{\mathcal{L}^{(2)}}.
\end{split}
\end{equation*}
This together the similar computations for $\partial_{\theta}W_1$ yields
\begin{equation}\label{W*2}
\Vert\nabla W_1\Vert_{L^{2}(\Omega_R)}\leq C\sigma\Big(\Vert(W^*_1,W^*_2)\Vert_{L^2(\Omega_R)}+\Vert (Z_1,Z_2,Z_3)\Vert_{L^2(\Omega_R)}
+\Vert Z_3(0,\cdot)\Vert_{L^2([-\theta_0,\theta_0])}\Big).
\end{equation}
One can deduce the same estimates in \eqref{W*1} and \eqref{W*2} for $W_2$. Therefore, it holds that
\begin{equation}\label{NNNH1}
\Vert (W_1,W_2)\Vert_{H^{1}(\Omega_R)}\leq C_4^*\sigma\Big(\Vert(W^*_1,W^*_2)\Vert_{L^2(\Omega_R)}+\Vert (Z_1,Z_2,Z_3)\Vert_{L^2(\Omega_R)}
+\Vert Z_3(0,\cdot)\Vert_{L^2([-\theta_0,\theta_0])}\Big),
\end{equation}
provided the constant $C_4^*$ depends only on the background data and $\epsilon_0$.
\par It remains to estimate $\Vert (Z_1,Z_2,Z_3)\Vert_{L^2(\Omega_R)}+\Vert Z_3(0,\cdot)\Vert_{L^2([-\theta_0,\theta_0])}$. Note that $(\mathcal{U}_i^Q,\mathcal{V}_i^Q,\check{\Phi}_i^Q)$ solves the nonlinear boundary value problem  \eqref{NNelli1} and \eqref{NNaBC1}, respectively. Then $(Z_1,Z_2,Z_3)$ satisfies the system
\begin{equation}
\begin{cases}
\begin{aligned}
&L^*_1(Z_1,Z_2,Z_3)=f^*_1(r,\theta;\mathcal{K}_i^*,\mathcal{U}_i^Q,\mathcal{V}_i^Q,\check{\Phi}_i^Q),\\
&L^*_2(Z_1,Z_3)=f^*_2(r,\theta;\mathcal{S}_i^*,\mathcal{K}_i^*,\mathcal{U}_i^Q,\mathcal{V}_i^Q,\check{\Phi}_i^Q),\\
&(\hat{r}Z_2)_r-\partial_{\theta}Z_1=f^*_3(r,\theta;\mathcal{S}_i^*,
\mathcal{K}_i^*,\mathcal{U}_i^Q,\mathcal{V}_i^Q,\check{\Phi}_i^Q),
\end{aligned}
\end{cases}
\end{equation}
in $\Omega_R$, subject to the boundary conditions
\begin{equation}
\begin{cases}
(Z_1,Z_2,\partial_rZ_3)(0,\theta)=(0,0,0),\ \ &{\rm{on}}\ \ \Gamma_{en},\\
Z_3(R,\theta)=0,\ \ &{\rm{on}}\ \ \Gamma_{ex},\\
Z_2(r,\pm\theta_0)=\partial_{\theta}Z_3(r,\pm\theta_0)=0, \ \ &{\rm{on}}\ \ \Gamma_w,
\end{cases}
\end{equation}
where
\begin{equation*}
\begin{split}
L^*_1(Z_1,Z_2,Z_3)=&
\hat{a}_{11}(r,\theta;\mathcal{K}^*_1,\mathcal{U}^Q_1,\mathcal{V}^Q_1,\check{\Phi}^Q_1)\partial_rZ_1
+\hat{a}_{22}(r,\theta;\mathcal{K}^*_1,\mathcal{U}^Q_1,\mathcal{V}^Q_1,\check{\Phi}^Q_1)\partial_{\theta}Z_2\\
&+\hat{a}_{12}(\mathcal{U}^Q_1,\mathcal{V}^Q_1)\partial_{\theta}Z_1
+\hat{a}_{21}(\mathcal{U}^Q_1,\mathcal{V}^Q_1)\partial_{r}Z_2
+\bar{a}_1(r)Z_1+\bar{b}_1(r)Z_3+\bar{b}_2(r)\partial_rZ_3,\\
L_2(Z_1,Z_3)=&(\hat r\partial_{r}Z_3)_r+\frac{1}{\hat{r}}\partial_{\theta\theta}Z_3
+\bar{c}_1(r)Z_1+\bar{c}_2(r)Z_3,
\end{split}
\end{equation*}
and
\begin{equation*}
\begin{aligned}
f^*_1
=&f_1(r,\theta;\mathcal{K}_1^*,\mathcal{U}_1^Q,\mathcal{V}_1^Q,\check{\Phi}_1^Q)
-f_1(r,\theta;\mathcal{K}_2^*,\mathcal{U}_2^Q,\mathcal{V}_2^Q,\check{\Phi}_2^Q)\\
&+\Big(\hat{a}_{11}(r,\theta;\mathcal{K}^*_2,\mathcal{U}^Q_2,\mathcal{V}^Q_2,\check{\Phi}^Q_2)
-\hat{a}_{11}(r,\theta;\mathcal{K}^*_1,\mathcal{U}^Q_1,\mathcal{V}^Q_1,\check{\Phi}^Q_1)\Big)\partial_{r}\mathcal{U}^Q_2\\
&+\Big(\hat{a}_{22}(r,\theta;\mathcal{K}^*_2,\mathcal{U}^Q_2,\mathcal{V}^Q_2,\check{\Phi}^Q_2)
-\hat{a}_{22}(r,\theta;\mathcal{K}^*_1,\mathcal{U}^Q_1,\mathcal{V}^Q_1,\check{\Phi}^Q_1)\Big)\partial_{\theta}\mathcal{V}^Q_2\\
&+\Big(\hat{a}_{12}(r,\theta;\mathcal{U}^Q_2,\mathcal{V}^Q_2)
-\hat{a}_{12}(r,\theta;\mathcal{U}^Q_1,\mathcal{V}^Q_1)\Big)\partial_{\theta}\mathcal{U}^Q_2\\
&+\Big(\hat{a}_{21}(r,\theta;\mathcal{U}^Q_2,\mathcal{V}^Q_2)
-\hat{a}_{21}(r,\theta;\mathcal{U}^Q_1,\mathcal{V}^Q_1)\Big)\partial_{r}
\mathcal{V}^Q_2,\\
f^*_2
=&f_2(r,\theta;\mathcal{S}_1^*,\mathcal{K}_1^*,\mathcal{U}_1^Q,\mathcal{V}_1^Q,\check{\Phi}_1^Q)
-f_2(r,\theta;\mathcal{S}_2^*,\mathcal{K}_2^*,\mathcal{U}_2^Q,\mathcal{V}_2^Q,\check{\Phi}_2^Q),\\
f^*_3
=&f_3(r,\theta;\mathcal{S}_1^*,\mathcal{K}_1^*,\mathcal{U}_1^Q,\mathcal{V}_1^Q,\check{\Phi}_1^Q)
-f_3(r,\theta;\mathcal{S}_2^*,\mathcal{K}_2^*,\mathcal{U}_2^Q,\mathcal{V}_2^Q,\check{\Phi}_2^Q).
\end{aligned}
\end{equation*}
It holds that
\begin{equation*}
\begin{split}
&\Vert (f^*_1,f^*_3)\Vert_{L^2(\Omega_R)}
\leq C \left(\Vert (W_1^*,W_2^*)\Vert_{L^2(\Omega_R)}+\delta_{\nu}\Vert (Z_1,Z_2,Z_3)\Vert_{L^2(\Omega_R)}\right),\\
&\Vert f^*_2\Vert_{L^2(\Omega_R)}
\leq C \left(\Vert (W_1^*,W_2^*)\Vert_{H^1(\Omega_R)}+\delta_{\mu}\Vert (Z_1,Z_2,Z_3)\Vert_{L^2(\Omega_R)}\right).
\end{split}
\end{equation*}
similar to the problem \eqref{Nelli1P} and \eqref{NaBC1P}, there holds
\begin{equation*}
\Vert (Z_1,Z_2)\Vert_{L^2(\Omega_R)}+\Vert Z_3\Vert_{H^1(\Omega_R)}
\leq C \left(\Vert (W_1^*,W_2^*)\Vert_{H^1(\Omega_R)}+(\delta_{\mu}+\delta_{\nu})\Vert (Z_1,Z_2,Z_3)\Vert_{L^2(\Omega_R)}\right).
\end{equation*}
Using the trace theorem deduces
\begin{equation*}
\Vert Z_3(0,\cdot)\Vert_{L^2([-\theta_0,\theta_0])}
\leq C \left(\Vert (W_1^*,W_2^*)\Vert_{H^1(\Omega_R)}+(\delta_{\mu}+\delta_{\nu})\Vert (Z_1,Z_2,Z_3)\Vert_{L^2(\Omega_R)}\right).
\end{equation*}
These together with \eqref{NNNnu} and \eqref{NNNmu} yield
\begin{equation}
\begin{split}
&\Vert (Z_1,Z_2,Z_3)\Vert_{L^2(\Omega_R)}+\Vert Z_3(0,\cdot)\Vert_{L^2([-\theta_0,\theta_0])}\\
&\leq C_5^* \left(\Vert (W_1^*,W_2^*)\Vert_{H^1(\Omega_R)}+(\delta_{\mu}+\delta_{\nu})\Vert (Z_1,Z_2,Z_3)\Vert_{L^2(\Omega_R)}\right)\\
&\leq C_5^* \left(\Vert (W_1^*,W_2^*)\Vert_{H^1(\Omega_R)}+\left(4C_3^*\sigma+4C_1^*(4C_3^*\sigma+\sigma)\right)\Vert (Z_1,Z_2,Z_3)\Vert_{L^2(\Omega_R)}\right),
\end{split}
\end{equation}
where the constant $C_5^*$ depends only on the background data and $\epsilon_0$. If $\sigma$ satisfies
\begin{equation}\label{Ndelta3}
\sigma\leq \frac{1}{2C_5^*\left(4C_3^*+4C_1^*(4C_3^*+1)\right)},
\end{equation}
then
\begin{equation*}
\begin{split}
\Vert (Z_1,Z_2,Z_3)\Vert_{L^2(\Omega_R)}+\Vert Z_3(0,\cdot)\Vert_{L^2([-\theta_0,\theta_0])}\leq 2C_5^*\Vert (W_1^*,W_2^*)\Vert_{H^1(\Omega_R)}.
\end{split}
\end{equation*}
Furthermore, \eqref{NNNH1} implies that
\begin{equation}
\Vert (W_1,W_2)\Vert_{H^{(1)}(\Omega_R)}\leq C_4^*(2C_5^*+1)\sigma (\Vert(W^*_1,W^*_2)\Vert_{H^1(\Omega_R)}.
\end{equation}
If $\sigma$ is chosen to satisfy
\begin{equation}\label{Ndelta4}
\sigma\leq \frac{1}{4C_4^*(2C_5^*+1)},
\end{equation}
and \eqref{Ndelta1}, \eqref{Ndelta2} \eqref{Ndelta3} hold, then the contraction mapping theorem implies that there exists a unique fixed point $(\mathcal{S}^\sharp,\mathcal{K}^\sharp)\in\mathcal{I}_{\delta_{\mu},R}$ for the iteration mapping $\mathcal{T}_2$.
\par In the following, we choose $\tilde{\sigma}_*$ such that, whenever $\sigma\in(0,\sigma_*]$, the conditions \eqref{NNNmunu}, \eqref{Ndelta1}, \eqref{Ndelta2}, \eqref{Ndelta3} and \eqref{Ndelta4} are all satisfied. Due to \eqref{NNNnu} and \eqref{NNNmu}, the condition \eqref{NNNmunu} holds if
\begin{equation*}
\sigma\leq \frac{\delta_0}{4C_3^*+4C_1^*\left(4C_3^*+1\right)}:=\sigma^1_*.
\end{equation*}
The conditions \eqref{Ndelta1} and \eqref{Ndelta2} hold if
\begin{equation*}
\sigma\leq \min\left\{\frac{1}{16(C_1^*)^2(4C_3^*+1)},\frac{1}{32C_2^*(C_1^*+1)(4C_3^*+1)}\right\}:=\sigma^2_*.
\end{equation*}
The conditions \eqref{Ndelta3} and \eqref{Ndelta4} hold if
\begin{equation*}
\sigma\leq \min\left\{\frac{1}{2C_5^*\left(4C_3^*+4C_1^*(4C_3^*+1)\right)},\frac{1}{4C_4^*\left(2C_5^*+1\right)}\right\}:=\sigma^3_*.
\end{equation*}
If $\tilde{\sigma}_*$ is chosen as
\begin{equation}
\tilde{\sigma}_*=\min\{\sigma^1_*,\sigma^2_*,\sigma^3_*\},
\end{equation}
then for any $\sigma\in(0,\tilde{\sigma}_*]$, the above conditions hold. \\
\textbf{Step 3: The uniqueness of the solution to Problem \ref{pro1}.} The arguments in Step 1 and Step 2 implies that there exists classical solutions $(\mathcal{S}^\sharp,\mathcal{K}^\sharp,\mathcal{U}^{Q^{\sharp}},\mathcal{V}^{Q^{\sharp}},\check{\Phi}^{Q^{\sharp}})$ to the nonlinear boundary value problem \eqref{elli1} and \eqref{aBC1}. That is, if $\sigma\in(0,\tilde{\sigma}_*]$, then the existence of the solutions to Problem \ref{pro1} is established. In the following, we assume that $(U_i,V_i,\Phi_i,\mathscr{K}_i,S_i)$ $(i=1,2)$ are two solutions to Problem \ref{pro1} and satisfy the estimate \eqref{sH4} and \eqref{sH4-1}. Set
\begin{equation*}
(\mathcal{M}_1,\mathcal{M}_2,\mathcal{M}_3,\mathcal{M}_4,\mathcal{M}_5)
:=(U_1,V_1,\Phi_1,\mathscr{K}_1,S_1)-(U_2,V_2,\Phi_2,\mathscr{K}_2,S_2).
\end{equation*}
Similar to the arguments in Step 1 and Step 2, we derive that there exists a constant depending only on the background data and $\epsilon_0$ such that
\begin{equation}
\Vert(\mathcal{M}_1,\mathcal{M}_2)\Vert_{L^2(\Omega_R)}
+\Vert(\mathcal{M}_3,\mathcal{M}_4,\mathcal{M}_5)\Vert_{H^1(\Omega_R)}\leq \mathcal{C}^*_6\sigma \left(\Vert(\mathcal{M}_1,\mathcal{M}_2)\Vert_{L^2(\Omega_R)}
+\Vert(\mathcal{M}_3,\mathcal{M}_4,\mathcal{M}_5)\Vert_{H^1(\Omega_R)}\right).
\end{equation}
If
\begin{equation*}
\sigma\leq \min\left\{\tilde{\sigma}_*,\frac{1}{2\mathcal{C}^*_6}\right\}:=\sigma_*,
\end{equation*}
then $(\mathcal{M}_1,\mathcal{M}_2,\mathcal{M}_3,\mathcal{M}_4,\mathcal{M}_5)=\mathbf{0}$, i.e. $(U_1,V_1,\Phi_1,\mathscr{K}_1,S_1)=(U_2,V_2,\Phi_2,\mathscr{K}_2,S_2)$. This completes the proof of Theorem \ref{thm1}.

{\bf Acknowledgement.}
 The research of Yuanyuan Xing is partially supported by  the Natural Science Foundation of Hebei province, China (No. A2025501003) and Fundamental Research Funds for the Central Universities (N2523027). The research of Zihao Zhang was supported by  the Postdoctoral Fellowship Program of CPSF under Grant Number GZB20250719.
\par {\bf Data availability.} No data was used for the research described in the article.
    \par {\bf Conflict of interest.} This work does not have any conflicts of interest.


\begin{thebibliography}{99}
\bibitem{Ascher}
U. M. Ascher, P. A. Markowich, P. Pietra, C. Schmeiser, {\it A phase plane analysis of transonic solutions for the hydrodynamic semiconductor model}, Math. Models Methods Appl. Sci., 1(3): 347-376, 1991.
\bibitem{Bae1}
M. Bae, B. Duan, J. J. Xiao, C. J. Xie, {\it Structural stability of supersonic solutions to the Euler-Poisson system}, Arch. Ration. Mech. Anal., 239(2): 679-731, 2021.
\bibitem{Bae2}
M. Bae, B. Duan, C. J. Xie, {\it Subsonic solutions for steady Euler-Poisson system in two-dimensional nozzles}, SIAM J. Math. Anal., 46(5): 3455-3480, 2014.
\bibitem{Bae3}
M. Bae, B. Duan, C. J. Xie, {\it Subsonic flow for the multidimensional Euler-Poisson system}, Arch. Ration. Mech. Anal., 220(1): 155-191, 2016.
\bibitem{Bae4}
M. Bae, B. Duan, C. J. Xie, {\it Two-dimensional subsonic flows with self-gravitation in bounded domain}, Math. Models Methods Appl. Sci., 25(14): 2721-2747, 2015.
\bibitem{Bae8}
M. Bae, B. Duan, C.J. Xie, {\it  Classical solutions to a mixed-type PDE with a Keldysh-type degeneracy and accelerating transonic solutions to the Euler-Poisson system}, arXiv:2505.16354.
\bibitem{Bae6}
M. Bae, H. Park, {\it Three-dimensional supersonic flows of Euler-Poisson system for potential flow}, Commun. Pure Appl. Anal., 20(7-8): 2421-2440, 2021.
\bibitem{Bae7}
M. Bae, Y. Park, {\it Radial transonic shock solutions of Euler-Poisson system in convergent nozzles}, Discrete Contin. Dyn. Syst. Ser. S, 11(5): 773-791, 2018.
\bibitem{Bae5}
M. Bae, S. K. Weng, {\it 3-D axisymmetric subsonic flows with nonzero swirl for the compressible Euler-Poisson system}, Ann. Inst. H. Poincar$\acute{\mathrm{e}}$ C Anal. Non Lin¨¦aire 35(1): 161-186, 2018.
\bibitem{Cao}
Y. Cao, Y. Y. Xing, {\it Subsonic Euler-Poisson flows with self-gravitation in an annulus}, J. Differential Equations, 411: 90-118, 2024.
\bibitem{Chen}
D. P. Chen, R. S. Eisenberg, J. W. Jerome, C.W. Shu, {\em A hydrodynamic model of temperature change in open ionic channels}, Biophys. J., 69: 2304-2322, 1995.
\bibitem{Duan0}
B. Duan, C. P. Wang, Y. Y. Xing, {\it Supersonic Euler-Poisson flows in divergent nozzles}, J. Differential Equations, 371: 598-628, 2023.
\bibitem{Duan}
B. Duan, Z. Luo, C. P. Wang, {\it Structural stability of non-isentropic Euler-Poisson system for gaseous stars}, J. Differential Equations, 417: 105-131, 2025.
\bibitem{Degond1}
P. Degond, P. A. Markowich, {\it On a one-dimensional steady-state hydrodynamic model for semiconductors}, Appl. Math. Lett., 3(3): 25-29, 1990.
\bibitem{Degond2}
P. Degond, P. A. Markowich, {\it A steady state potential flow model for semiconductors}, Ann. Mat. Pura Appl., 165(4): 87-98, 1993.
\bibitem{Gamba}
I.M. Gamba, {\it Stationary transonic solutions of a one-dimensional hydrodynamic model for semiconductors}, Comm. Partial Differential Equations, 17(3-4): 553-577, 1992.
\bibitem{Gilbarg}
D. Gilbarg, N. S. Trudinger, {\it Elliptic partial differential equations of second order}, Grundlehren Math. Wiss., Springer, Berlin, 224, 1998.
\bibitem{Guo}
Y. Guo, W. Strauss, {\it Stability of semiconductor states with insulating and contact boundary conditions}, Arch. Ration. Mech. Anal., 179(1): 1-30, 2006.
\bibitem{Liu}
T. P. Liu, {\it Nonlinear stability and instability of transonic flows through a nozzle}, Comm. Math. Phys., 83(2): 243-260, 1982.
\bibitem{Luo1}
T. Luo, Z. P. Xin, {\it Transonic shock solutions for a system of Euler-Poisson equations}, Commun. Math. Sci., 10(2): 419-462, 2012.
\bibitem{Luo2}
T. Luo, J. Rauch, C. J. Xie, Z. P. Xin, {\it Stability of transonic shock solutions for one-dimensional Euler-Poisson equations}, Arch. Ration. Mech. Anal., 2011, 202(3): 787-827.
\bibitem{Markowich0}
P. A. Markowich, {\it On steady state Euler-Poisson models for semiconductors}, Z. Angew. Math. Phys., 42(3): 389-407, 1991.
\bibitem{Markowich}
P. A. Markowich, C. A. Ringhofer, C. Schmeiser, {\it Semiconductor Equations}, Springer Verlag, Vienna, 1990.
\bibitem{Rosini}
M. D. Rosini, {\it A phase analysis of transonic solutions for the hydrodynamic semiconductor model}, Quart. Appl. Math., 63(2): 251-268, 2005.
\bibitem{WZ25}
 C. P. Wang,  Z. H. Zhang, {\it Structural stability of smooth axisymmetric subsonic spiral flows with self-gravitation in a concentric cylinder}, J. Differential Equations 438:  113356, 2025.
 \bibitem{WZ25-1}
  C. P. Wang,  Z. H. Zhang, {\it Structural stability of cylindrical supersonic solutions to the steady Euler¨CPoisson system}, J. Lond. Math. Soc. (2) 112 (2025), no. 1, Paper No. e70229.
  \bibitem{WZ25-2}
  C. P. Wang,  Z. H. Zhang, {\it Structural stability of supersonic spiral flows with large angular velocity for the Euler-Poisson system}, arXiv:2505.12195.
\bibitem{Weng}
S. K. Weng, {\it A deformation-curl-Poisson decomposition to the three dimensional steady Euler-Poisson system with applications}, J. Differential Equations, 267(11): 6574-6603, 2019.
\bibitem{XZ25}
  Y. Y. Xing,  Z. H. Zhang, {\it
Subsonic Euler-Poisson flows with nonzero vorticity in convergent nozzles}, arXiv:2505.19032.
\bibitem{Yeh}
L. M. Yeh, {\it On a steady state Euler-Poisson model for semiconductors}, Comm. Partial Differential Equations, 21(7-8), 1007-1034, 1996.
\end{thebibliography}
\end{document}